\documentclass[a4paper, 11pt]{article}

\usepackage[T1]{fontenc}
\usepackage[utf8]{inputenc}
\usepackage[english]{babel}
\usepackage{microtype}
\usepackage{float}
\usepackage{graphicx}
\usepackage{subcaption}
\usepackage{enumitem}

\usepackage{dsfont}
\usepackage{hyperref}

\usepackage{amsmath, amsthm, amssymb, amsfonts, mathtools, mathrsfs}
\usepackage{bbm}

\usepackage{tikz}
\usetikzlibrary{arrows}
\usetikzlibrary{calc}
\usetikzlibrary{shapes}
\usetikzlibrary{positioning}
\usetikzlibrary{shapes.geometric}
\usetikzlibrary{arrows.meta}
\usetikzlibrary{patterns,patterns.meta}
\usetikzlibrary{decorations.pathreplacing,calligraphy,decorations.pathmorphing}

\theoremstyle{plain}
\newtheorem{lemma}{Lemma}
\newtheorem{theorem}[lemma]{Theorem}
\newtheorem*{thmnonumber}{Theorem 1.4}
\newtheorem{proposition}[lemma]{Proposition}

\newtheorem{corollary}[lemma]{Corollary}
\newtheorem{remark}[lemma]{Remark}
\newtheorem{question}[lemma]{Question}
\newtheorem*{defi}{Definition}
\newtheorem*{notation}{Notation}

\numberwithin{equation}{section}
\numberwithin{lemma}{section}

\newcommand{\norm}[1]{\left\lVert#1\right\rVert}
\DeclareMathOperator*{\esssup}{ess\,sup}
\DeclareMathOperator*{\essinf}{ess\,inf}

\newcommand{\upin}{\rotatebox[origin=c]{-90}{$\supset$}}

\begin{document}
\title{Limits of action convergent graph sequences\\with unbounded $(p,q)$-norms}
\author{Aranka Hrušková\\
\normalsize \textit{Alfréd Rényi Insitute of Mathematics, Reáltanoda utca 13-15, Budapest, Hungary}}
\date{}

\maketitle
\begin{abstract}
The recently developed notion of action convergence by Backhausz and Szegedy \cite{action} unifies and generalises the dense (graphon) and local-global (graphing) convergences of graph sequences. This is done through viewing graphs as operators and examining their dynamical properties.
Suppose $(A_n)_n^\infty$ is a sequence of operators representing graphs, Cauchy with respect to the action metric. If $(A_n)_n^\infty$ has uniformly bounded $(p,q)$-norms where $(p,q)$ is any pair in $[1,\infty)\times(1,\infty)$, then Backhausz and Szegedy prove that $(A_n)_n^\infty$ has a limit operator which, moreover, must be self-adjoint and positivity-preserving. In the present work, we construct a large class of graph sequences whose only uniformly bounded $(p,q)$-norm is the $(\infty,1)$-norm, but which converge nonetheless. We show that the limit operators in this case are not unique, not self-adjoint, and need not be positivity-preserving. In particular, in the action convergence language, this means that the space of graphops is not compact. By identifying these multiple limits, we also demonstrate that $c$-regularity is not invariant under weak equivalence, where $c$ is the eigenvalue of the identity function, when the identity function is an eigenfunction.
\end{abstract}

\section{Introduction}

The central object of the relatively young field of graph limit theory is a sequence $(G_n)_{n=1}^\infty$ of finite graphs, for which we seek to find a limit object. The two most thoroughly developed convergence notions, based on different methods of sampling small subgraphs from large graphs, are, however, only applicable when either the number of edges in $G_n$ is asymptotically quadratic in terms of the number of vertices \cite{BorgsChayesLovSosVesz, Lov, LovaszSzegedy06DenseLimits} or when the maximum degree $\Delta(G_n)$ is uniformly bounded above \cite{Lov, Remco}. We call the former sequences \emph{dense} and have \emph{graphons} for them as limit objects, while the latter are an extreme case of \emph{sparse} sequences and their limits are \emph{graphings} \cite{Hatami2014}. This leaves out the territory of sequences in which the number of edges grows subquadratically but superlinearly in terms of the number of vertices -- that includes for example the hypercubes, the incidence graphs of finite projective planes, and many regimes of the Erdős-Rényi model $\mathcal{G}(n,p(n))$. A number of authors have in recent years defined various extensions of the classical notions mentioned above with the aim of reaching the world of sequences with intermediate densities \cite{BorgsChayesHoldenExchange, BorgsChayesCohnZhaoLpII, BorgsChayesCohnZhaoLpI, sconvDolezal, Frenkel, s-graphons, SparsityNesetMend, Szegedy2015logconvergence}. In this paper, we are interested in \emph{action convergence} introduced by Backhausz and Szegedy \cite{action}, which unifies and generalises graphons and graphings in the common framework of \emph{$P$-operators}.

\begin{defi}[$P$-operator]
Let $(\Omega,\mathcal{A},\mu)$ be a probability space. Then we call a linear operator $A\colon L^\infty(\Omega)\to L^1(\Omega)$ a P-operator if the norm
\[
\norm{A}_{\infty\to1}=\sup_{v\in L^\infty(\Omega)}\frac{\norm{Av}_1}{\norm{v}_\infty}
\]
is finite.
\end{defi}

As the name and the definition of the limit object above suggest, action convergence lays emphasis on the dynamical properties of graphs. There is a number of natural operators associated to a finite graph, of which the most prominent are the adjacency matrix, the discrete Laplacian, and the transition matrix of the simple random walk; in this paper, we will always identify a finite graph with its adjacency matrix. The key to the definition of the convergence is the notion of $\emph{profiles}$ that allow us to compare operators even when they act on $L$-spaces of different probability spaces. A $k$-profile $\mathcal{S}_k(A)$ of a $P$-operator $A$ is a collection of probability measures on $\mathbb{R}^{2k}$ encoding the actions of $A$ and giving rise to the following metrisation.

\begin{defi}
Let $A\colon L^\infty(\Omega_A)\to L^1(\Omega_A)$ and $B\colon L^\infty(\Omega_B)\to L^1(\Omega_B)$ be $P$-operators. Then their action distance is
\[
d_M(A,B)=\sum_{k=1}^\infty\frac{d_H(\mathcal{S}_k(A),\mathcal{S}_k(B))}{2^k},
\]
where $d_H$ is the Hausdorff distance. A sequence $(A_n)_n^\infty$ of $P$-operators \emph{action converges} to a $P$-operator $A$ if and only if $\lim_nd_M(A_n,A)=0$. When $d_M(A,B)=0$, we say that $A$ and $B$ are \emph{weakly equivalent}.
\end{defi}

Importantly, Backhausz and Szegedy show in \cite{action} that a sequence $(G_n)_n^\infty$ of graphs with uniformly bounded maximum degree converges locally-globally to a graphing $\mathcal{G}$ if and only if their adjacency matrices $\left(A(G_n)\right)_n^\infty$ action converge to $\mathcal{G}$, and that a sequence $(G_n)_n^\infty$ of graphs converges to a graphon $W$ if and only if their scaled adjacency matrices $\left(\frac{A(G_n)}{|V(G_n)|}\right)_n^\infty$ action converge to $W$.
In our setting, a graphon $W\colon[0,1]^2\to[0,1]$ becomes the $P$-operator that sends $f\in L^2\left([0,1],\lambda\right)$ to
\[
(Wf)(x)=\int_0^1W(x,y)f(y) \, d\lambda(y),
\]
and a graphing $(\mathcal{G},\nu)$ becomes the $P$-operator that sends $v\in L^2\left(V(\mathcal{G}),\nu\right)$ to
\[
(\mathcal{G}v)(x)=\sum_{xy\in E(\mathcal{G})} v(y).
\]
In particular, both graphons and graphings are not just operators from $L^\infty$ to $L^1$, but from $L^2$ to $L^2$, and satisfy that $\norm{W}_{2\to2}\leq1$ and $\norm{\mathcal{G}}_{2\to2}\leq d$, where $d$ is the maximum degree of the graphing $\mathcal{G}$. Both of them are also self-adjoint and positivity-preserving -- any $P$-operator that has these two properties is called a $\emph{graphop}$.
Given a sequence of $P$-operators, we can in general deduce information about the existence of its limit and about the limit's properties if we can assume uniform boundedness of some $(p,q)$-norms like the $(2,2)$-norms above.

\begin{lemma}[Sequential compactness, Lemma 2.6 in \cite{action}]
Let $(A_n)_{n=1}^\infty$ be a sequence of $P$-operators with uniformly bounded $\norm{\cdot}_{\infty\to1}$ norms.
Then $(A_n)_{n=1}^\infty$ has a Cauchy subsequence with respect to the distance $d_M$.
\end{lemma}

\begin{theorem}[Existence of limit object, Theorem 2.9 in \cite{action}]\label{thm:limit}
Let $p\in[1,\infty)$ and $q\in[1,\infty]$. Let $(A_n)_{n=1}^\infty$ be a sequence of $P$-operators, Cauchy with respect to the distance $d_M$ and with uniformly bounded $\norm{\cdot}_{p\to q}$ norms. Then there is a $P$-operator $A$ such that $\lim_nd_M(A_n,A)=0$ and $\norm{A}_{p\to q}\leq\limsup_n\norm{A_n}_{p\to q}$.
\end{theorem}

For $c\in\mathbb{R}$, a $P$-operator is called $c$-regular if the identity function $\mathds{1}$ is an eigenfunction with eigenvalue $c$, i.e., $A\mathds{1}=c\mathds{1}$.

\begin{proposition}[Section 3 in \cite{action}]\label{prop:graphops_remain_graphops}
Let $p,q\in[1,\infty]$ be fixed and let $(A_n)_n^\infty$ be a sequence of $P$-operators with uniformly bounded $(p,q)$-norms. Suppose that $(A_n)_n^\infty$ action converges to a $P$-operator $A$. Then
\begin{enumerate}[label=(\alph*)]
    \item if $q\notin\{1,\infty\}$ and $A_n$ is self-adjoint for every $n$, then $A$ is also self-adjoint,
    \item if $p\neq\infty$ and $A_n$ is positivity-preserving for every $n$, then $A$ is also positivity-preserving, and    
    \item if $p\neq\infty$, $c\in\mathbb{R}$ and $A_n$ is $c$-regular for every $n$, then $A$ is also $c$-regular.
\end{enumerate}
\end{proposition}

In the present work, we are interested in optimality of the restrictions on $p$ and $q$ in Theorem~\ref{thm:limit} and Proposition~\ref{prop:graphops_remain_graphops}. We quickly show that part (a) of Proposition~\ref{prop:graphops_remain_graphops} must hold for all $(p,q)\in[1,\infty]^2\setminus\{(\infty,1)\}$,
but our main result is identifying a large class of sequences $(A_n)_n^\infty$ in which $A_n$ is self-adjoint and positivity-preserving for every $n$ and whose only uniformly bounded $(p,q)$-norms are the $(\infty,1)$-norms, but which at the same time have multiple limit objects, none of which are self-adjoint and some of which are not positivity-preserving. In other words, we show that a sequence of graphops does not necessarily action converge to a graphop.
The graphops in our sequences arise as the adjacency matrices of finite simple graphs $G$ that contain a vertex of degree $|V(G)|-1$, that is a vertex adjacent to every other vertex of $G$. We call $G^+$ the graph on $|V(G)|+1$ vertices obtained from $G$ by adding a vertex like that.

\begin{theorem}\label{thm:TheTheorem}
Let $(G_n)_n^\infty$ be a sequence of finite graphs with $|V(G_n)|\to\infty$, whose adjacency operators action converge to a $P$-operator $A\colon L^\infty(\Omega,\nu)\to L^1(\Omega,\nu)$, where $(\Omega,\nu)$ is separable.
Then there is a $\nu$-filter $\mathcal{F}$ on $\Omega$ such that for any $\nu$-ultrafilter $\mathcal{U}$ extending $\mathcal{F}$, both
\begin{equation*}
\begin{aligned}[c]
A^+\colon L^\infty(\Omega,\nu)&\to L^1(\Omega,\nu)\\
\text{given by }\left(A^+g\right)(\omega)&=(Ag)(\omega)+\phi_\mathcal{U}(g)
\end{aligned}
\qquad\text{ and }\qquad
\begin{aligned}[c]
A^-\colon L^\infty(\Omega,\nu)&\to L^1(\Omega,\nu)\\
\text{given by }\left(A^-g\right)(\omega)&=(Ag)(\omega)-\phi_\mathcal{U}(g)
\end{aligned}
\end{equation*}
are action limits of $(G_n^+)_n^\infty$, where $\phi_\mathcal{U}\colon L^\infty(\Omega,\nu)\to\mathbb{R}$ is the functional sending $\lim_{n\to\infty}\sum_{i=1}^n\alpha_{n,i}\chi_{E_{n,i}}$ to $\lim_{n\to\infty}\sum_{i=1}^n\alpha_{n,i}\mathds{1}_{E_{n,i}\in\mathcal{U}}$.
\end{theorem}

In the construction of the limit objects, we utilise a functional $\phi_\mathcal{U}$ from the dual space $(L^\infty)^*$ of $L^\infty$ which are not in the canonical embedding of $L^1$ in $(L^\infty)^*$. In particular, $\phi_\mathcal{U}\colon L^\infty\to\mathbb{R}$ arises from a finitely additive measure given by an ultrafilter $\mathcal{U}$, where $\mathcal{U}$ is any extension of a suitable countably generated filter $\mathcal{F}$. In the course of proving that such a filter $\mathcal{F}$ exists, we establish Theorem~\ref{thm:FilterMagic} which we believe is of independent interest. It states that for \emph{any} linear operator $A\colon L^\infty(\Omega)\to L^1(\Omega)$, where $\Omega$ is a finite measure space, there is an ultrafilter $\mathcal{U}$ such that for any function $f\in L^\infty$ and a number $a\in\mathbb{R}$, we can find functions $f_a$ which are arbitrarily close to $f$ in the 1-norm and also whose images under $A$ are arbitrarily close to $Af$ in the 1-norm, but such that $\phi_\mathcal{U}(f_a)=a$, independently of the value of $\phi_\mathcal{U}(f)$.

The rest of the paper is organised as follows. In Section~\ref{section:preliminaries}, we define $k$-profiles and prove the extension of part (a) of Proposition~\ref{prop:graphops_remain_graphops}.
In Section~\ref{section:stars}, we warm up with the most basic case of our construction which is the star graphs $(S_n)_n^\infty$. We recall the structure of $(L^\infty)^*$, explain how the functional $\phi_\mathcal{U}$ plays the role of the high-degree vertex, prove that the $(\infty,1)$-norm is not continuous with respect to $d_M$ by showing that $\lim_n\norm{S_n}_{\infty\to1}>\norm{\lim_n S_n}_{\infty\to1}$, establish that no action limit of $S_n$ can be self-adjoint, and prove a special case of Theorem \ref{thm:TheTheorem}.
The main results are in Section~\ref{section:main} in which we first describe what the limiting $k$-profiles of $\left(G_n^+\right)_n^\infty$ must be. Then we prove Theorem~\ref{thm:FilterMagic} and relying on it, we go on to prove Theorem~\ref{thm:TheTheorem}.
We close off in Section~\ref{section:closeoff} with two open questions.

\section{Preliminaries}\label{section:preliminaries}

Let $A\colon L^\infty(\Omega)\to L^1(\Omega)$ be a $P$-operator. Then for every $f\in L^\infty(\Omega)$, the pair $(f,Af)$ represents an observation of the dynamical properties of $A$. We want to compress this observation into a form which will not involve $\Omega$ and also neglect some inessential features of $A$ (see Remark ...), and do so by considering the law of the random variable
\begin{align*}
(f,Af)\colon & \Omega\to\mathbb{R}^2\\
 & \omega\mapsto(f(\omega),Af(\omega)).
\end{align*}
More generally, taking a $k$-tuple $(f_1,\dots,f_k)$ of functions in $L^\infty(\Omega)$ provides an even finer observation $(f_1,\dots,f_k,Af_1,\dots,Af_k)$, and by compressing and collecting all of these, we arrive at a set of probability measures which captures the various ways in which $A$ acts on functions.
\begin{defi}[$k$-profile]
    Let $A\colon L^\infty(\Omega)\to L^1(\Omega)$ be a $P$-operator and $k$ a positive integer. The $k$-profile $\mathcal{S}_k(A)$ of $A$ is
    \[\mathcal{S}_k(A):=\left\{\mathcal{D}(f_1,\dots,f_k,Af_1,\dots,Af_k) : f_1,\dots,f_k\in B_1^{L^\infty}\right\}\subset\mathcal{P}\left(\mathbb{R}^{2k}\right),\]
    where $B_1^{L^\infty}$ is the closed unit ball of $L^\infty(\Omega)$ and $\mathcal{D}(f_1,\dots,f_k,Af_1,\dots,Af_k)$ is the joint distribution of $f_1,\dots,f_k$, $Af_1,\dots,Af_k$, i.e., the image measure given by the map $\omega\mapsto\left(f_1(\omega),\dots,f_k(\omega),Af_1(\omega),\dots,Af_k(\omega)\right)$.
\end{defi}

We will also be using the shorthand
\[
\mathcal{D}_A(f_1,\dots,f_k):=\mathcal{D}\left(f_1,\dots,f_k,Af_1,\dots,Af_k\right).
\]

For a hands-on example, when $A$ is an $n\times n$ matrix then its $k$-profile is the set of all discrete probability measures on $\mathbb{R}^{2k}$ of the form
\[
\frac{1}{n}\sum_{j=1}^n\delta_{((v_1)_j,\dots,(v_k)_j,(v_1A)_j,\dots,(v_kA)_j)}
\]
where $v_1,\dots,v_k$ are real vectors with entries in $[-1,1]$ and $\delta_x$ denotes the Dirac measure concentrated on $x\in\mathbb{R}^{2k}$. (We assumed the uniform distribution on $[n]$ here.)

As explained in the Introduction, we measure the similarity of $k$-profiles of different $P$-operators by the Hausdorff distance.

\begin{defi}[Hausdorff pseudometric]
Let $(M,d)$ be a metric space. Then the Hausdorff pseudometric $d_H$ on the power set $\mathscr{P}(M)$ is given by
\[d_H(X,Y)=\max\left\{\sup_{x\in X}\inf_{y\in Y} d(x,y),\sup_{y\in Y}\inf_{x\in X} d(x,y)\right\}\]
for all subsets $X$, $Y$ of $M$.
\end{defi}
Note that $d_H(X,Y)=0$ if and only if $\overline{X}=\overline{Y}$.

It remains to determine what metric we consider on the set $\mathcal{P}\left(\mathbb{R}^{2k}\right)$ of probability measures on $\mathbb{R}^{2k}$. Since we want to metrise the weak convergence of measures, we choose the Lévy-Prokhorov metric.
\begin{defi}[Lévy-Prokhorov metric]
Let $(M,d)$ be a metric space, $\mathcal{B}(M)$ its associated Borel $\sigma$-algebra, and $\mathcal{P}(M)$ the set of all probability measures on $(M,\mathcal{B}(M))$. The Lévy-Prokhorov metric $d_{LP}$ on $\mathcal{P}(M)$ is given by
\[
d_{LP}(\eta,\mu)=\inf\{\varepsilon>0 :
\eta(U)\leq\mu(U^\varepsilon)+\varepsilon
\text{ and }
\mu(U)\leq\eta(U^\varepsilon)+\varepsilon
\text{ for every Borel set } U\subset M
\},\]
where $U^\varepsilon=\{x\in M : d(x,U)<\varepsilon\}$.
\end{defi}

Note that for any two probability measures $\eta,\mu\in\mathcal{P}(\mathbb{R}^n)$, we have $d_{LP}(\eta,\mu)\leq1$, hence also for any two $P$-operators $A, B$, their distance satisfies $d_M(A,B)\leq\sum_{k=1}^\infty\frac{1}{2^k}=1$.

Let now $(\cdot,\cdot)_A$ denote the bilinear form on functions from $L^\infty(\Omega)$ given by a $P$-operator $A$ as follows:
\[
(f,g)_A:=\int_\Omega (Af)g \,d\mu=\mathbb{E}[(Af)g].
\]
\begin{defi}
A $P$-operator $A\colon L^\infty(\Omega)\to L^1(\Omega)$ is
\begin{enumerate}
    \item \emph{self-adjoint} if $(f,g)_A=(g,f)_A$ for all $g,f\in L^\infty(\Omega)$,
    \item \emph{positivity-preserving} if $(Af)(\omega)\geq0$ holds for a.a. $\omega\in\Omega$ whenever $f(\omega)\geq0$ holds for a.a. $\omega\in\Omega$,
    \item \emph{$c$-regular} if $A\mathds{1}=c\mathds{1}$, for $c\in\mathbb{R}$,
    \item a \emph{graphop} if it is self-adjoint and positivity-preserving.
\end{enumerate}
\end{defi}

While the definition of a $P$-operator requires the $(\infty,1)$-operator norm to be finite, the general definition of a $(p,q)$-norm, for $p,q\in[1,\infty]$, is
\[
\norm{A}_{p\to q}:=\sup_{f\in L^\infty}\frac{\norm{Af}_q}{\norm{f}_p}.
\]
Since the $q$-norms $\norm{\cdot}_q$ are increasing with $q$, the operator norms $\norm{\cdot}_{p\to q}$ are increasing with $q$ and decreasing with $p$, meaning in particular that
\[\norm{A}_{\infty\to1}\leq\norm{A}_{p\to q}\]
for any linear operator $A$ and any $p,q\in[1,\infty]$.

Following the lines of the argument made in \cite{inducedpqNorm} for finite matrices, we quickly prove the following.

\begin{lemma}\label{lemma:adjointnorms}
Let $p, q$ be in $[1,\infty]$, and let $p', q'$ be their Hölder conjugates. Let $A$ and $A^*$ be $P$-operators satisfying that
\[(v,w)_A=(w,v)_{A^*} \text{ for all }v,w\in L^\infty.\]
Then $\norm{A}_{p\to q}=\norm{A^*}_{q'\to p'}$.
\end{lemma}
\begin{proof}
\begin{align*}
\norm{A}_{p\to q}
=\sup_{f\in L^\infty}\frac{\norm{Af}_q}{\norm{f}_p}
&=\sup_{f\in L^\infty}\left\{\norm{Af}_q : \norm{f}_p\leq1\right\}\\
&=\sup_{f\in L^\infty}\left\{\sup_{g\in L^\infty}\left\{\left|\int(Af)g\right| : \norm{g}_{q'}\leq1\right\} : \norm{f}_p\leq1\right\}\\
&=\sup_{g\in L^\infty}\left\{\sup_{f\in L^\infty}\left\{\left|\int f(A^*g)\right| : \norm{f}_{p}\leq1\right\} : \norm{g}_{q'}\leq1\right\}\\
&=\sup_{g\in L^\infty}\left\{\norm{A^*g}_{p'} : \norm{g}_{q'}\leq1\right\}\\
&=\sup_{g\in L^\infty}\frac{\norm{A^*g}_{p'}}{\norm{g}_{q'}}\\
&=\norm{A^*}_{q'\to p'}
\end{align*}
\end{proof}

\begin{corollary}\label{coroll:selfadjoint_just_infty1}
The assumption in Proposition~\ref{prop:graphops_remain_graphops} (a) can be extended from $(p,q)\in[1,\infty]\times(1,\infty)$ to $(p,q)\in[1,\infty]^2\setminus\{(\infty,1)\}$.
\end{corollary}
\begin{proof}
    Let $p\in(1,\infty)$.
    If $\norm{A}_{p\to1}$ or $\norm{A}_{p\to\infty}$ are uniformly bounded then so are $\norm{A}_{\infty\to p'}$ or $\norm{A}_{1\to p'}$ by Lemma~\ref{lemma:adjointnorms}, and so Proposition~\ref{prop:graphops_remain_graphops} (a) gives that $A$ is self-adjoint.

    If $\norm{A}_{1\to1}$, $\norm{A}_{1\to\infty}$ or $\norm{A}_{\infty\to\infty}$ are uniformly bounded then so are, say,
    \[\norm{A}_{2\to1}=\norm{A}_{\infty\to2},\]
    by monotonicity of the $(p,q)$-norms. Now we can apply Proposition~\ref{prop:graphops_remain_graphops} (a) to $\norm{A}_{\infty\to2}$ to conclude that $A$ is self-adjoint.
\end{proof}

Finally, as remarked in \cite[Section 2]{action}, the $(p,q)$-norms can be read out from the 1-profiles of $P$-operators (and are hence invariant under weak equivalence). For a measure $\mu$ on $\mathbb{R}^2$, let $\mu_y\in\mathcal{P}(\mathbb{R})$ denote its $y$-axis marginal. Then in the case of the $(\infty,1)$-norm, we can get it as follows:
\begin{align*}
    \infty>\norm{B}_{\infty\to1}=\sup_{f\in L^\infty}\frac{\norm{Bf}_1}{\norm{f}_\infty}
    &=\sup\left\{\norm{fB}_1 : \norm{f}_\infty\leq1\right\}\\
    &=\sup\left\{\int_\Omega |Bf|\, d\nu : f\in B_1^{L^\infty}\right\}\\
    &=\sup\left\{\int_\mathbb{R} |x|\, d\mathcal{D}(Bf) : f\in B_1^{L^\infty}\right\}\\
    &=\sup\left\{\int_\mathbb{R} |x|\, d \mu_y(x): \mu\in \mathcal{S}_{1}(B)\right\}.
\end{align*}
This enables us to prove that the $(\infty,1)$-norms are lower semicontinuous with respect to action convergence.
\begin{lemma}
Let $(A_n)_n^\infty$ be a Cauchy sequence of $P$-operators with uniformly bounded $\norm{\cdot}_{\infty\to1}$ norms. Suppose further that $\lim_{n\to\infty}d_M(A_n,A)=0$ for a $P$-operator $A$. Then $\norm{A}_{\infty\to1}\leq\liminf_{n\to\infty}\norm{A_n}_{\infty\to1}$.
\end{lemma}
\begin{proof}
Suppose that a sequence $(\mu_n)_n^\infty$ of measures on $\mathbb{R}^2$ converges in $d_{LP}$ to a measure $\mu$. Then for every $n$, $d_{LP}\left((\mu_n)_y,\mu_y\right)\leq d_{LP}(\mu_n,\mu)$, and so the measures $(\mu_n)_y$ converge in $d_{LP}$ to $\mu_y$. The Lévy-Prokhorov distance is on $\mathbb{R}$ a metrisation of weak convergence of measures, and so by the portmanteau theorem,\[\int_{\mathbb{R}}|x|\, d\mu_y(x)\leq\liminf_{n\to\infty}\int_\mathbb{R}|x| \, d(\mu_n)_y(x)\]
because the function $|x|$ is continuous and bounded below.
This implies both that the equality
\[\norm{B}_{\infty\to1}=\sup\left\{\int_\mathbb{R} |x|\, d \mu_y(x): \mu\in \mathcal{S}_{1}(B)\right\}\]
from above can be extended to
\[\norm{B}_{\infty\to1}=\sup\left\{\int_\mathbb{R} |x|\, d \mu_y(x): \mu\in \mathcal{S}_{1}(B)\right\}=\sup\left\{\int_\mathbb{R} |x|\, d \mu_y(x): \mu\in \overline{\mathcal{S}_{1}(B)}\right\}\]
and that any measure $\mu\in\mathcal{S}_1(A)$ and any sequence $(\mu_n)_n^\infty$ with $\mu_n\in\mathcal{S}_1(A_n)$ and $\mu_n\xrightarrow{d_{LP}}\mu$ must satisfy
\[\int_\mathbb{R}|x| \, d\mu_y(x)\leq\liminf_{n\to\infty}\int_\mathbb{R}|x| \, d(\mu_n)_y(x)\leq\liminf_{n\to\infty}\norm{A_n}_{\infty\to1},\]
implying
\[\norm{A}_{\infty\to1}=\sup\left\{\int_\mathbb{R} |x|\, d \mu_y(x): \mu\in \mathcal{S}_{1}(B)\right\}\leq\liminf_{n\to\infty}\norm{A_n}_{\infty\to1}.\]
\end{proof}
However, we will see in Section~\ref{section:stars} that the $(\infty,1)$-norm is not continuous.

\section{The case of stars}\label{section:stars}
The question naturally arises whether Theorem~\ref{thm:limit} is as good as it gets, that is whether there are Cauchy sequences with uniformly bounded norms $\norm{\cdot}_{\infty\to q}$ for some $q\in[1,\infty]$ which, however, do not action converge to any $P$-operator.

For positivity-preserving operators $A$,
\[\norm{A}_{\infty\to q}=\norm{A\mathds{1}}_{q},\]
where $\mathds{1}=\chi_{\Omega}$ is the constant function with value 1. At the same time, we have seen that for self-adjoint operators $A$,
\[\norm{A}_{\infty\to q}=\norm{A}_{q'\to1}.\]
Suppose then that we want to look for a candidate Cauchy sequence $(A_n)_n^\infty$ which would exemplify the impossibility of extending Theorem~\ref{thm:limit} beyond $p\neq\infty$ -- if we want to find it among graphops then these two observations tell us that $(\norm{A_n\mathds{1}}_q)_n^\infty$ must be unbounded for every $q\in(1,\infty]$.
Let us note that if $A$ is the adjacency matrix of a graph then $A\mathds{1}=\underline{d}$ is (an ordering of) its degree sequence, and so the sequence $(S_n)_n^\infty$ of stars, having as large a difference between minimum and maximum degree as possible in a simple graph, becomes in immediate candidate to investigate.

In particular, we denote by $S_n$ the $n$-vertex tree with $n-1$ leaves. Then we can check that for $q\in(1,\infty)$, the $q$-norm of its degree sequence is unbounded:

\[
\norm{\underline{d}}_q=\left(\frac{1}{n}\left(1^q+\dots+1^q+(n-1)^q\right)\right)^{\frac{1}{q}}
=\left(\frac{n-1}{n}+\frac{(n-1)^q}{n}\right)^{\frac{1}{q}}\geq\frac{n-1}{n^{\frac{1}{q}}}\to\infty \text{ as } n\to\infty.
\]
By the monotonicities of $(p,q)$-norms, this means that the $(\infty,1)$-norm
\[\norm{S_n}_{\infty\to1}=\frac{1}{n}(1\times(n-1)+(n-1)\times1)<2\]
is the \emph{only} uniformly bounded $(p,q)$-norm of the star sequence, where we took the liberty of identifying $S_n$ with its adjacency matrix.

In the rest of the section, we will investigate what a putative action limit $A$ of $(S_n)_n^\infty$ would have to satisfy, showing that necessarily, the continuity of the $(\infty,1)$-norm cannot hold, i.e.,
\[\lim_{n\to\infty}\norm{S_n}_{\infty\to1}\neq\norm{A}_{\infty\to1},\]
and that $A$ cannot be self-adjoint, so in particular, it cannot be a graphop. We will derive an action limit at the end of the section.

\begin{lemma}\label{lemma:closure}
A $P$-operator
$A$ satisfies $\lim_{n\to\infty} d_M(S_n,A)=0$ if and only if
\[\overline{\mathcal{S}_k(A)}
=\mathcal{P}\left([-1,1]^k\right)\times
\left\{\delta_z : z\in[-1,1]^k\right\}\]
for every positive integer $k$.
\end{lemma}
We shall in fact deduce this lemma from the following more general result.

\begin{lemma}[Uniform approximability of $\mathcal{P}(M)$]\label{lemma:uniformapprox}
Let $(M,d)$ be a totally bounded metric space.
Then for all $\varepsilon>0$, there is $N=N(\varepsilon)$ such that for every $\mu\in\mathcal{P}(M,\mathcal{B}(M))$, there is a sequence $\big(\mu_n=\frac{1}{n}\sum_{i=1}^n\delta_{x_i}\big)_n^\infty$ of discrete probability measures such that
\[\forall n\geq N(\varepsilon)\ \ d_{LP}(\mu,\mu_n)<\varepsilon.\]
\end{lemma}
In particular, $N$ does not depend on $\mu$.
\begin{proof}
For a positive integer $k$, let us fix a finite set $\{z_1,\dots,z_{m(k)}\}\subseteq M$ whose $\frac{1}{k}$-balls cover $M$ and a partition $M_1,\dots,M_{m(k)}$ in $\mathcal{B}(M)$ of $M$ such that $z_i\in M_i\subseteq B_{1/k}(z_i)$ for every $i\in[m]=[m(k)]$. Given a probability measure $\mu$ on $M$ and a partition $M_1,\dots,M_m$ as described, let $\kappa_m\in\mathcal{P}(M)$ be $\kappa_m=\sum_i^m\mu(M_i)\delta_{z_i}$. Then for every measurable $A$,
\[
\mu(A)=\sum_{i=1}^m\mu(A\cap M_i)
=\sum_{\substack{i\in[m] \\ A\cap M_i\neq\emptyset}}\mu(A\cap M_i)\leq\sum_{\substack{i\in[m] \\ A\cap M_i\neq\emptyset}}\mu(M_i)
=\sum_{\substack{i\in[m] \\ A\cap M_i\neq\emptyset}}\kappa_m(\{z_i\})\leq\kappa_m\left(A^{1/k}\right)
\]
and
\[
\kappa_m(A)=\sum_{\substack{i\in[m] \\ z_i\in A}}\kappa_m(\{z_i\})=\sum_{\substack{i\in[m] \\ z_i\in A}}\mu(M_i)\leq\mu(A^{1/k}),
\]
so $d_{LP}\left(\mu,\kappa_m\right)\leq\frac{1}{k}$.

The next step is to approximate the discrete measure $\kappa_m$ by $\mu_n=\frac{1}{n}\sum_i^n\delta_{x_i}$ where for every $i\in[n]$ we will put $x_i=z_j$ for some $j\in[m]$. Let $\mu_n$ be any probability measure on $\{z_1,\dots,z_m\}$ such that $\mu_n\left(\{z_j\}\right)=\left \lfloor{n\kappa_m(z_j)}\right \rfloor/n$ or $\left \lceil{n\kappa_m(z_j)}\right \rceil/n$ for every $j\in[m]$. Then $|\kappa_m(z_j)-\mu_n(z_j)|<\frac{1}{n}$ for all $j\in[m]$, and so $d_{LP}(\kappa_m,\mu_n)<\frac{m}{n}$. Finally, the triangle inequality implies
\[
d_{LP}(\mu,\mu_n)\leq d_{LP}(\mu,\kappa_m)+d_{LP}\left(\kappa_m,\mu_n\right)<\frac{1}{k}+\frac{m}{n},
\]
 and hence for every $\mu\in\mathcal{P}(M)$ and $n\geq m(k)\cdot k$, we have that $d_{LP}(\mu,\mu_n)<2/k$ as desired.
\end{proof}

\begin{proof}[Proof of Lemma~\ref{lemma:closure}]
Let us first note that $\lim_{n\to\infty} d_M(S_n,A)=0$ if and only if
\[\lim_{n\to\infty} d_H\left(\mathcal{S}_k(S_n),\mathcal{S}_k(A)\right)=0\] for every positive integer $k$, so let us fix $k$ for the rest of the proof and show that $\lim_{n\to\infty} d_H\left(\mathcal{S}_k(S_n),\mathcal{S}_k(A)\right)=0$ if and only if
\[\overline{\mathcal{S}_k(A)}
=\mathcal{P}\left([-1,1]^k\right)\times
\left\{\delta_z : z\in[-1,1]^k\right\}.\]

Lemma~\ref{lemma:uniformapprox} implies that for every
$\varepsilon>0$, there is an
$N=N(\varepsilon)$ such that $\forall n\geq N$,
\[
d_H\left(
\Big\{\frac{1}{n}\sum_{i=1}^n\delta_{x_i} : x_i\in[-1,1]^k\Big\}\times\big\{\delta_z : z\in[-1,1]^k\big\},
\mathcal{P}\left([-1,1]^k\right)\times\big\{\delta_z : z\in[-1,1]^k\big\}
\right)
\leq\varepsilon.
\]
On the other hand, $S_n$ sends the vector $(a_1,a_2,\dots,a_n)\in[-1,1]^n$ to $(\sum_{j=2}^na_j,a_1,\dots,a_1)$, which means that every choice of vectors $v_1,\dots,v_k$ with entries in $[-1,1]$ gives an element
$\frac{1}{n}\big(
\delta_{((v_1)_1,\dots,(v_k)_1,\sum_{j=2}^n(v_1)_j,\dots,\sum_{j=2}^n(v_k)_j)}+\sum_{j=2}^n\delta_{((v_1)_j,\dots,(v_k)_j,(v_1)_1,\dots,(v_k)_1)}\big)$ of $\mathcal{S}_k(S_n)$ which satisfies
\begin{align*}
d_{LP}\bigg(
\frac{1}{n}\Big(
\delta_{((v_1)_1,\dots,(v_k)_1,\sum_{j=2}^n(v_1)_j,\dots,\sum_{j=2}^n(v_k)_j)}+&\sum_{j=2}^n\delta_{((v_1)_j,\dots,(v_k)_j,(v_1)_1,\dots,(v_k)_1)}\Big),\\
\frac{1}{n}\Big(\delta_x+&\sum_{j=2}^n\delta_{((v_1)_j,\dots,(v_k)_j)}\Big)\times\delta_{((v_1)_1,\dots,(v_k)_1)}
\bigg)\leq\frac{1}{n}
\end{align*}
for every $x\in[-1,1]^k$.
This means that
\[d_H\left(\mathcal{S}_k(S_n),\left\{\frac{1}{n}\sum_{i=1}^n\delta_{x_i} : x_i\in[-1,1]^k\right\}\times\left\{\delta_z : z\in[-1,1]^k\right\}\right)\leq\frac{1}{n},\]
and so the triangle inequality tells us that $\forall n\geq N$,
\begin{align*}
&&&&d_H\big(
\mathcal{S}_k(S_n),&\mathcal{P}\left([-1,1]^k\right)\times\big\{\delta_z : z\in[-1,1]^k\big\}
\big)\\
&&\leq &&
d_H\Big(\mathcal{S}_k(S_n),&\Big\{\frac{1}{n}\sum_{i=1}^n\delta_{x_i} : x_i\in[-1,1]^k\Big\}\times\big\{\delta_z : z\in[-1,1]^k\big\}\Big)\\
&&&&+\ d_H\Big(&\Big\{\frac{1}{n}\sum_{i=1}^n\delta_{x_i} : x_i\in[-1,1]^k\Big\}\times\left\{\delta_z : z\in[-1,1]^k\right\},\mathcal{P}\left([-1,1]^k\right)\times\left\{\delta_z : z\in[-1,1]^k\right\}\Big)\\
&&\leq & & \varepsilon\ +\ \frac{1}{n}.&
\end{align*}

Now if
$\overline{\mathcal{S}_k(A)}=\mathcal{P}\left([-1,1]^k\right)\times\{\delta_z : z\in[-1,1]^k\}$
then $\forall n\geq N$,
\begin{align*}
& d_H\big(\mathcal{S}_k(S_n),\mathcal{S}_k(A)\big)\\
\leq\ & d_H\left(\mathcal{S}_k(S_n),\mathcal{P}\left([-1,1]^k\right)\times\{\delta_z : z\in[-1,1]^k\}\right)+d_H\left(\mathcal{P}\left([-1,1]^k\right)\times\{\delta_z : z\in[-1,1]^k\},\mathcal{S}_k(A)\right)\\
=\ & d_H\left(\mathcal{S}_k(S_n),\mathcal{P}\left([-1,1]^k\right)\times\{\delta_z : z\in[-1,1]^k\}\right)\\
\leq\ & \varepsilon+\frac{1}{n},
\end{align*}
and so $\lim_{n\to\infty} d_H(\mathcal{S}_k(S_n),\mathcal{S}_k(A))=0$.

On the other hand, if $\lim_n d_H(\mathcal{S}_k(S_n),\mathcal{S}_k(A))=0$ and $\varepsilon>0$ then there is $M=M(\varepsilon)$ such that $\forall n\geq M$ $d_H(\mathcal{S}_k(S_n),\mathcal{S}_k(A))<\varepsilon$.
Combining this with the computations above, we get that $\forall n\geq\max\{N(\varepsilon),M(\varepsilon)\}$,
\begin{align*}
&d_H\big(\mathcal{S}_k(A),\mathcal{P}([-1,1]^k)\times\{\delta_z : z\in[-1,1]^k\}\big)\\
\leq\ & d_H\left(\mathcal{S}_k(A),\mathcal{S}_k(S_n)\right)+d_H\big(\mathcal{S}_k(S_n),\mathcal{P}([-1,1]^k)\times\{\delta_z : z\in[-1,1]^k\}\big)\\
<\ &2\varepsilon+\frac{1}{n},
\end{align*}
and so in fact
$d_H\left(\mathcal{S}_k(A),\mathcal{P}([-1,1]^k)\times\{\delta_z : z\in[-1,1]^k\}\right)=0$. As $\mathcal{P}\left([-1,1]^k\right)\times\left\{\delta_z : z\in[-1,1]^k\right\}$ is closed, this is equivalent to $\overline{\mathcal{S}_k(A)}=\mathcal{P}\left([-1,1]^k\right)\times\left\{\delta_z : z\in[-1,1]^k\right\}$ as desired.
\end{proof}

Let us now assume that there is a $P$-operator $A\colon L^\infty(\Omega,\mu)\to L^1(\Omega,\mu)$ for some probability space $(\Omega,\mathcal{A},\mu)$ such that $\lim_{n\to\infty}d_M(S_n,A)=0$.

Every $f\in B_1^{L^\infty}$ gives an element $\mathcal{D}(f,Af)$ of $\mathcal{S}_1(A)$ which, by Lemma~\ref{lemma:uniformapprox}, is of the form $\nu\times\delta_{c_f}$ for some constant $c_f\in[-1,1]$. That is, $A$ is in fact an element of $L^\infty(\Omega)^*$ and sends $f$ to $c_f$.

\begin{lemma}\label{cislessthanfinfty}
For every $f\in B_1^{L^\infty}$, it must be the case that $|c_f|\leq\norm{f}_\infty$.
\end{lemma}
\begin{proof}
\begin{align*}
\norm{\frac{f}{\norm{f}_\infty}}_\infty=1
\Rightarrow\ &\frac{f}{\norm{f}_\infty}\in B_1^{L^\infty}\\
\Rightarrow\ &\left\lvert A\frac{f}{\norm{f}_\infty}\right\rvert\equiv\lvert c_{\frac{f}{\norm{f}}}\rvert\leq1\\
\Rightarrow\ &\norm{f}_\infty\left\lvert A\frac{f}{\norm{f}_\infty}\right\rvert\leq\norm{f}_\infty\\
\Rightarrow\ &|c_f|\equiv\left\lvert Af\right\rvert\leq\norm{f}_\infty\\
\end{align*}
\end{proof}

Next we note that for every $x\in[-1,1]\setminus\{0\}$, the Dirac measure $\delta_{(0,x)}$ is in $\overline{\mathcal{S}_1(A)}$, while the lemma above says it cannot be in the profile $\mathcal{S}_1(A)$ itself. Thus for every $x\in[-1,1]\setminus\{0\}$ and for every $\epsilon>0$ there must be $f_{x,\epsilon}\in B_1^{L^\infty}$ such that $d_{LP}\big(\delta_{(0,x)},\mathcal{D}_A(f_{x,\epsilon})\big)<\epsilon$. Now
\begin{align}
&d_{LP}\big(\delta_{(0,x)},\mathcal{D}_A(f_{x,\epsilon})\big)<\epsilon \nonumber \\
\Leftrightarrow\ 
&\inf\{\delta>0 :
\delta_{(0,x)}(U)\leq\mathcal{D}_A(f_{x,\epsilon})(U^\delta)+\delta
\text{ and }
\mathcal{D}_A(f_{x,\epsilon})(U)\leq\delta_{(0,x)}(U^\delta)+\delta
\text{ for every Borel set } U\subset\mathbb{R}^2
\}<\epsilon \nonumber \\
\Leftrightarrow\ 
&\delta_{(0,x)}(U)\leq\mathcal{D}_A(f_{x,\epsilon})(U^\epsilon)+\epsilon
\text{ and }
\mathcal{D}_A(f_{x,\epsilon})(U)\leq\delta_{(0,x)}(U^\epsilon)+\epsilon
\text{ for every Borel set } U\subset\mathbb{R}^2 \nonumber \\
\Leftrightarrow\ 
&1\leq\mathcal{D}_A(f_{x,\epsilon})\big(B_\epsilon((0,x))\big)+\epsilon \nonumber \\
\Rightarrow\ 
&c_{f_{x,\epsilon}}\in(x-\epsilon,x+\epsilon)\cap[-1,1].\label{eqn:c_f}
\end{align}
Also let $\Omega_{<\epsilon}^x$ be the subset $\{\omega : |f_{x,\epsilon}(\omega)|<\epsilon\}$ of $\Omega$, and similarly let $\Omega_{\geq\epsilon}^x=\{\omega : |f_{x,\epsilon}(\omega)|\geq\epsilon\}\subset\Omega$.
Then apart from the line (\ref{eqn:c_f}) above, $1\leq\mathcal{D}_A(f_{x,\epsilon})\big(B_\epsilon((0,x))\big)+\epsilon$ also implies that $\mu\left(\Omega_{<\epsilon}^x\right)\geq1-\epsilon$ and $\mu\left(\Omega_{\geq\epsilon}^x\right)\leq\epsilon$.

\begin{corollary}\label{coroll:norm=1}
If a $P$-operator $A$ is an action limit of $(S_n)_n^\infty$ then $\norm{A}_{\infty\to1}=1$.
\end{corollary}
\begin{proof}
\[
\norm{A}_{\infty\to1}
=\sup_{f\in B_1^{L^\infty}}\frac{\norm{Af}_1}{\norm{f}_\infty}
=\sup_{f\in B_1^{L^\infty}}\frac{\lvert c_f\rvert}{\norm{f}_\infty}
\leq\sup_{f\in B_1^{L^\infty}}\frac{\norm{f}_\infty}{\norm{f}_\infty}
=1.
\]
On the other hand, for every $\varepsilon>0$
\[
\norm{A}_{\infty\to1}
\geq\frac{\norm{Af_{1,\varepsilon}}_1}{\norm{f_{1,\varepsilon}}_\infty}
=\frac{\lvert c_{f_{1,\varepsilon}}\rvert}{\norm{f_{1,\varepsilon}}_\infty}
>\frac{1-\varepsilon}{\norm{f_{1,\varepsilon}}_\infty}
\geq1-\varepsilon,
\]
so $\norm{A}_{\infty\to1}\geq1$, completing the proof.
\end{proof}

Since we will later construct an action limit of $(S_n)_n^\infty$, Corollary~\ref{coroll:norm=1} shows that, as claimed at the beginning of the section, the $(\infty,1)$-norm is not continuous with respect to action convergence.
In particular, if $\lim_n d_M(S_n,A)=0$ then
\[\lim_{n\to\infty}\norm{S_n}_{\infty\to1}=\lim_{n\to\infty}\frac{2n-2}{n}=2\neq1=\norm{A}_{\infty\to1}.\]

Using the notation set up before Corollary~\ref{coroll:norm=1}, we will now deliver on our second promise of proving that $A$ cannot possibly be self-adjoint.
\begin{proposition}\label{prop:not_selfadjoint}
    Suppose that a $P$-operator $A$ is an action limit of $(S_n)_n^\infty$. Then $A$ is not self-adjoint.
\end{proposition}
\begin{proof}
    Let $\varepsilon\in(0,1/3)$ and $\mathds{1}$ be the characteristic function $\chi_\Omega$. Then
    \[
    \left|(\mathds{1},f_{1,\varepsilon})_A\right|
    =\left|\int(A\mathds{1})f_{1,\varepsilon} \,d\mu\right|
    =\left|c_\mathds{1}\int f_{1,\varepsilon} \,d\mu\right|
    \leq\left| \int_{\Omega^1_{<\varepsilon}} f_{1,\varepsilon} \,d\mu+\int_{\Omega^1_{\geq\varepsilon}} f_{1,\varepsilon} \,d\mu \right|
    \leq\varepsilon\mu\left(\Omega^1_{<\varepsilon}\right)+\mu\left(\Omega^1_{\geq\varepsilon}\right)
    \leq2\varepsilon
    \]
    while
    \[\left|(f_{1,\varepsilon},\mathds{1})_A\right|
    =\left|\int(Af_{1,\varepsilon})\mathds{1} \,d\mu\right|
    =|c_{f_{1,\varepsilon}}|>1-\varepsilon.\]
    But then
    \[\left|(f_{1,\varepsilon},\mathds{1})_A\right|>1-\varepsilon>2\varepsilon\geq\left|(\mathds{1},f_{1,\varepsilon})_A\right|,\]
    and so $(f_{1,\varepsilon},\mathds{1})_A\neq(\mathds{1},f_{1,\varepsilon})_A$ and $A$ is not self-adjoint.
\end{proof}

Again, since we will actually show the existence of an action limit $A$ at the end of the section, Proposition~\ref{prop:not_selfadjoint} demonstrates a limitation of Proposition~\ref{prop:graphops_remain_graphops} (a) and together with Corollary~\ref{coroll:selfadjoint_just_infty1} effectively answers the question of necessary self-adjointness of a limit of a sequence of self-adjoint $P$-operators.

\begin{proposition}\label{notinL1}
There is no $g\in L^1(\Omega,\mu)$ such that $Af\equiv\int_\Omega gf \text{d}\mu$ for every $f\in L^\infty(\Omega,\mu)$.
\end{proposition}
\begin{proof}

Suppose that on the contrary, $g\in L^1(\Omega,\mu)$ is as stated.
Then
\[
\norm{g}_1=\int_\Omega|g|\text{d}\mu=\int_\Omega g\left(\mathds{1}_{\{\omega : g(\omega)\geq0\}}-\mathds{1}_{\{\omega : g(\omega)<0\}}\right) \text{d}\mu
=A\left(\mathds{1}_{\{\omega : g(\omega)\geq0\}}-\mathds{1}_{\{\omega : g(\omega)<0\}}\right)\leq1.
\]
Now let $x\in[-1,1]\setminus\{0\}$ and $\epsilon>0$. By the discussion between Lemma~\ref{cislessthanfinfty} and Corollary~\ref{coroll:norm=1}, we have that
\[
x-\epsilon
<\lvert c_{f_{x,\epsilon}}\rvert
=\left\lvert\int gf_{x,\epsilon}\text{d}\mu\right\rvert\leq\int \lvert gf_{x,\epsilon}\rvert
=\int_{\Omega_{<\epsilon}^x}\lvert gf_{\epsilon}\rvert+\int_{\Omega_{\geq\epsilon}^x}\lvert gf_{\epsilon}\rvert
<\epsilon\int_{\Omega_{<\epsilon}^x}\lvert g\rvert+\int_{\Omega_{\geq\epsilon}^x}\lvert g\rvert
\leq\epsilon\Big(1-\int_{\Omega_{\geq\epsilon}^x}\lvert g\rvert\Big)+\int_{\Omega_{\geq\epsilon}^x}\lvert g\rvert,
\]
and so
\[
\int_{\Omega_{\geq\epsilon}^x}\lvert g\rvert>\frac{x-2\epsilon}{1-\epsilon}.
\]
Taking $x=1$, we get
\begin{equation}\label{eqn:intto1}
\int_{\Omega_{\geq\epsilon}^1}\lvert g\rvert>\frac{1-2\epsilon}{1-\epsilon}\to1\text{ as }\epsilon\to0,
\end{equation}
and recalling $\int_{\Omega_{\geq\epsilon}^1}\lvert g\rvert\leq\int_{\Omega}\lvert g\rvert\leq1$, we conclude that $\int_{\Omega}\lvert g\rvert=1$.

Now for any integer $n\geq2$, let $B_n=\bigcup_{i=n}^\infty\Omega_{\geq2^{-i}}^1\subset\Omega$. Note that for every $m\geq n$, (\ref{eqn:intto1}) tells us that
\[
1\geq\int_{B_n}|g|\geq\int_{\Omega_{\geq2^{-m}}^1}|g|>1-2^{1-m}\to1\text{ as }m\to\infty,
\]
and so for every $n\geq2$ we have $\int_{B_n}|g|=1$ and thus $\int_{\Omega\setminus B_n}|g|=0$. On the other hand, observe that
\[
\mu(B_n)=\mu\left(\bigcup_{i=n}^\infty\Omega_{\geq2^{-i}}^1\right)
\leq\sum_{i=n}^\infty\mu\left(\Omega_{\geq2^{-i}}^1\right)\leq\sum_{i=n}^\infty2^{-i}=2^{1-n},
\]
which implies that the intersection of the chain
$B_2\supseteq B_3\supseteq B_4\supseteq B_5\supseteq\dots$ has measure zero. But then
\[
1=\int_\Omega|g|=\int_{\cap_{n=2}^\infty B_n}|g|+\int_{\cup_{n=2}^\infty(\Omega\setminus B_n)}|g|=0+0=0,
\]
which is a contradiction.
\end{proof}

We have just shown that when $A\colon L^\infty(\Omega,\mathcal{A})\to\mathbb{R}$ is an action limit of $(S_n)_n^\infty$ then
\[A\in L^\infty(\Omega)^*\setminus L^1(\Omega),\]
where we are abusing notation and identifying elements $g\in L^1(\Omega)$ with the functionals $f\mapsto\int fg \,d\mu.$
Let us now recall the structure of the dual space $L^\infty(\Omega,\mathcal{A},\mu)^*$. We mostly follow the terminology and notation of \cite{TheDual} and \cite{hitchhiker}. Let $ba(\mathcal{A})$ be the space of of bounded finitely additive signed measures, also known as signed charges. The total variation of a charge $\nu$ on the $\sigma$-algebra $\mathcal{A}$ is
\[|\nu|(\Omega)=\sup\left\{\sum_{i=1}^n|\nu(M_i)| : \{M_1,\dots,M_n\}\text{ is a measurable partition of }\Omega\right\},\]
and $\nu\in ba(\mathcal{A})$ if and only if $|\nu|(\Omega)<\infty$.
The dual of $L^\infty(\mathcal{A},\mu)$ is represented by the subset $ba(\mathcal{A},\mu)$ of $ba(\mathcal{A})$ which consists of all the finitely additive signed measures $\nu$ that satisfy
\[\nu(N)=0 \text{ whenever } \mu(N)=0,\ \ \ \text{ for all }N\in\mathcal{A}.\]
Any such charge $\nu$ then gives a functional via
\[f\mapsto\int_\Omega f \,d\nu = \lim_{n\to\infty}\sum_{i=1}^n\alpha_{n,i}\nu\left(E_{n,i}\right),\]
where $\left(\sum_{i=1}^n\alpha_{n,i}\chi_{E_{n,i}}\right)_n^\infty$ is any sequence of simple functions which converge to $f$ in the $L^\infty$-norm,
and every element of $L^\infty(\mathcal{A},\mu)^*$ arises this way. Moreover, as shown for example in Section 6.2 of \cite{hitchhiker}, the $(\infty,1)$-norm of this functional is $|\nu|(\Omega)$. Having proved Corollary~\ref{coroll:norm=1}, we could have used this fact to streamline the first part of the proof of Proposition~\ref{notinL1}.

Every $\nu\in ba(\mathcal{A})$ can be uniquely written as $\nu=\kappa_\nu+\gamma_\nu$, where $\kappa_\nu$ is countably additive, i.e., $\kappa_\nu$ is a singed measure, and $\gamma_\nu$ is purely finitely additive. Having proved that if $A$ is an action limit of stars, the purely finitely additive part of the finitely additive measure representing $A$ must be non-trivial, we now zoom in on a particular subset of purely finitely additive measures. We start by introducing $\mu$-ultrafilters, which turn out to be in one-to-one correspondence with the $\{0,1\}$-valued elements of $ba(\mathcal{A},\mu)$.

\begin{defi}[$\mu$-filter]
Let $(\Omega,\mathcal{A},\mu)$ be a measure space. A non-empty collection $\mathcal{F}$ of measurable subsets of $\Omega$ is a \emph{$\mu$-filter} if and only if
\begin{itemize}
    \item $\mu(S)>0$ for all $S\in\mathcal{F}$,
    \item if $S,T\in\mathcal{F}$ then $S\cap T\in\mathcal{F}$ , and
    \item if $S\in\mathcal{F}$ and $S\subset T$ then $T\in\mathcal{F}$.
\end{itemize}
A maximal filter is called an \emph{ultrafilter}.
\end{defi}
Every $\mu$-ultrafilter $\mathcal{U}$ now gives a finitely additive measure $\delta_\mathcal{U}$ by setting
\[
\delta_\mathcal{U}(E)=\begin{cases}
        1, & \text{if } E\in\mathcal{U}\\
        0, & \text{if } E\notin\mathcal{U}.
        \end{cases}
\]
Moreover, every $\{0,1\}$-valued element of $ba(\mathcal{A},\mu)$ arises this way, and if $\mu$ is atomless then each such charge is purely finitely additive.

\begin{theorem}\label{thm:limitobject}
Let $\mathcal{U}$ be a $\lambda$-ultrafilter on the probability space $(\Omega=[0,1),\mathcal{B}([0,1)),\lambda)$.
Then the $P$-operator $A$ given by
\begin{align*}
A\colon L^\infty(\lambda)&\to L^1(\lambda)\\
f&\mapsto\left(\int f \,d\delta_\mathcal{U}\right)\cdot\chi_\Omega
\end{align*}
satisfies $\lim_{n\to\infty} d_M(S_n,A)=0$.
\end{theorem}
\begin{proof}
By Lemma~\ref{lemma:closure}, this is equivalent to showing that for every positive integer $k$,
\[\overline{\mathcal{S}_k(A)}
=\mathcal{P}\left([-1,1]^k\right)\times
\{\delta_z : z\in[-1,1]^k\}.\]
Let $k$ be fixed. The operator $A$ sends every element $f\in L^\infty(\Omega)$ to a constant function where the absolute value of the constant is at most $\norm{f}_\infty$, so in particular elements of $B_1^{L^\infty}$ are sent to $[-1,1]$. Thus $\mathcal{S}_k(A)\subset\mathcal{P}\left([-1,1]^k\right)\times
\{\delta_z : z\in[-1,1]^k\}$.
On the other hand, for every $\mu\in\mathcal{P}\left([-1,1]^k\right)$ and $z\in[-1,1]^k$, we will now construct a sequence $(\mathcal{D}_A(f_{n,1},\dots,f_{n,k}))_n^\infty$ in $\mathcal{S}_k(A)$ such that $\lim_n d_{LP}\left(\mu\times\delta_z,\mathcal{D}_A(f_{n,1},\dots,f_{n,k})\right)=0$, implying that $\mathcal{P}\left([-1,1]^k\right)\times
\{\delta_z : z\in[-1,1]^k\}\subseteq\overline{\mathcal{S}_k(A)}$.

We start by fixing a sequence $(E_n)_n^\infty$ in $\mathcal{U}$ such that for every $n\geq1$, $E_n=\big[\frac{i-1}{n},\frac{i}{n}\big)$ for some $i\in[n]$. Given $\mu\in\mathcal{P}\left([-1,1]^k\right)$ and $z\in[-1,1]^k$, let $(\mu_n=\frac{1}{n}\sum_i^n\delta_{x_i})_n^\infty$ be a sequence given by Lemma~\ref{lemma:uniformapprox}. Now for every $n\geq1$ and $j\in[k]$ let
\[
f_{n,j}(\omega)=
\begin{cases}
(x_i)_j, & \text{ if } \omega\in\big[\frac{i-1}{n},\frac{i}{n}\big)\setminus E_n\\
z_j, & \text{ if } \omega\in E_n.
\end{cases}
\]
Then
\begin{align*}
d_{LP}\big(\mu_n\times\delta_z,\mathcal{D}_A(f_{n,1},\dots,f_{n,k})\big)
&=d_{LP}\bigg(\mu_n\times\delta_z,\frac{1}{n}\Big(\delta_z+
\sum_{\substack{i\in[n]\\
[\frac{i-1}{n},\frac{i}{n})\neq E_n}}\delta_{x_i}\Big)\times\delta_z\bigg)\\
&=d_{LP}\bigg(\frac{1}{n}\sum_{i=1}^n\delta_{x_i},\frac{1}{n}\Big(\delta_z+
\sum_{\substack{i\in[n]\\
[\frac{i-1}{n},\frac{i}{n})\neq E_n}}\delta_{x_i}\Big)\bigg)\leq\frac{1}{n}.
\end{align*}
Hence, by the triangle inequality,
\begin{align*}
d_{LP}\big(\mu\times\delta_z,\mathcal{D}_A(f_{n,1},\dots,f_{n,k})\big)
&\leq d_{LP}\left(\mu\times\delta_z,\mu_n\times\delta_z\right)+d_{LP}\big(\mu_n\times\delta_z,\mathcal{D}_A(f_{n,1},\dots,f_{n,k})\big)\\
&\leq d_{LP}(\mu,\mu_n)+\frac{1}{n}
\to0 \text{ as } n\to\infty.
\end{align*}
\end{proof}

\begin{remark}
Note that the limit object in Theorem~\ref{thm:limitobject} is definitely not unique. Not only do we have a choice of the ultrafilter, but we could also take any $(M,\mathcal{B}(M),\kappa)$ instead of $(\Omega=[0,1),\mathcal{B}([0,1)),\lambda)$, where $(M,d)$ is a totally bounded metric space such that for infinitely many positive integers $n$, it is possible to partition $M$ into $M_1,\dots,M_n\in\mathcal{B}(M)$ with $\kappa(M_i)=\frac{1}{n}$ for all $i\in[n]$.
\end{remark}

\section{Graphs $G^+$ with a vertex that neighbours everything}\label{section:main}

Stars are in fact just a special case of the following construction.
\begin{defi}
Let $G$ be a graph. Then $G^+$ is the graph on $|V(G)|+1$ vertices formed from $G$ by adding a single vertex $v$ and connecting it to all the other vertices.
\end{defi}
Seen like this, a star with $n$ leaves is $E_n^+$ where $E_n$ is the empty graph on $n$ vertices. This viewpoint allows us to extend the result from the previous section to more sequences of the form $(G_n^+)_n^\infty$ than just $(S_n)_n^\infty=(E_n^+)_n^\infty$. Before stating the full theorem, we set up the scene with a couple of lemmas and propositions.

\begin{lemma}\label{lemma:verticesToInfinity}
Let $(G_n)_n^\infty$ be Cauchy in $d_M$.
If $\limsup_{n\to\infty}|V(G_n)|=\infty$ then $\liminf_{n\to\infty}|V(G_n)|=\infty$,
and for any action limit $A\colon L^\infty(\Omega,\nu)\to L^1(\Omega,\nu)$ of $(G_n)_n^\infty$, $(\Omega,\nu)$ is atomless.
\end{lemma}
\begin{proof}
Suppose on the contrary that there is a natural number $k$ such that for infinitely many $n$, the graph $G_n$ has $k$ vertices. Then there is a subsequence $(G_{m_j})_j$ such that $|V(G_{m_i})|=k$ for all $i\geq1$. On the other hand, let $(G_{n_i})_i$ be a subsequence with strictly increasing number of vertices.
For a measure $\mu\in\mathcal{P}\left(\mathbb{R}^2\right)$, let $\mu_x\in\mathcal{P}(\mathbb{R})$ be its $x$-axis marginal. Then $d_{LP}\left(\mu_x,\kappa_x\right)\leq d_{LP}(\mu,\kappa)$ for all $\mu,\kappa\in\mathcal{P}(\mathbb{R}^2)$, and thus
\[
d_H\left((\mathcal{S}_1)_x(G_n),(\mathcal{S}_1)_x(G_m)\right)\leq d_H\left(\mathcal{S}_1(G_n),\mathcal{S}_1(G_m)\right)\leq 2d_M(G_n,G_m)
\]
for any $n,m\in\mathbb{N}$, where
\[(\mathcal{S}_1)_x(G):=\left\{\mu_x : \mu\in\mathcal{S}_1(G)\right\}=\left\{\mathcal{D}(f) : f\in B_1^{L^\infty}\right\}.\]
Now $\lim_i (\mathcal{S}_1)_x(G_{n_i})$ contains the uniform measure on $[-1,1]$, which implies that both $\lim_j (\mathcal{S}_1)_x(G_{m_j})$ and $\overline{(\mathcal{S}_1)_x(A)}$ must contain it too. But $\lim_j (\mathcal{S}_1)_x(G_{m_j})$ only contains measures which are approximable arbitrarily well by $k$ atoms, so we get a contradiction with $\liminf_n|V(G_n)|<\infty$. Similarly, $\overline{(\mathcal{S}_1)_x(A)}$ containing the uniform measure on $[-1,1]$ implies that $(\Omega,\nu)$ is atomless.
\end{proof}

Lemma~\ref{lemma:verticesToInfinity} implies that if $(G_n^+)_n^\infty$ is Cauchy in $d_M$ and $\limsup_{n\to\infty}|V(G_n)|=\infty$ then the special added vertices $v_n$ will have smaller and smaller weight, yet the fact that they are adjacent to all the other vertices in their graph means that the value at $v_n$ has a great impact on the outcome after applying the adjacency operator. In particular, if two functions $f,g:V(G_n^+)\to\mathbb{R}$ only differ in the value they assign to the added vertex $v_n$, then as $n$ grows larger, the resulting measures $\mathcal{D}_{G_n}(f)$ and $\mathcal{D}_{G_n}(g)$ will more and more look like they are equivalent up to a shift along the $y$-axis by $f(v_n)-g(v_n)$. To formalise the intuition forming upon this observation, we introduce the following definition.

\begin{defi}
Let $\mu$ be a measure on $\mathbb{R}^n$ and $v\in\mathbb{R}^n$ a vector. Then $\mu\oplus v$ is the measure on $\mathbb{R}^n$ such that
\[(\mu\oplus v)(T)=\mu(T\oplus\{-v\})=\mu(\{t-v : t\in T\})\]
for any measurable subset $T$.
For any set $\mathcal{S}$ of measures on $\mathbb{R}^n$ and set $V\subseteq\mathbb{R}^n$ of vectors,
\[\mathcal{S}\oplus V:=\{\mu\oplus v : \mu\in\mathcal{S},v\in V\}.\]
\end{defi}

So as not to clutter the exposition with technicalities, we will only focus on sequences $(G_n)_n^\infty$ with $|V(G_n)|\to\infty$.

\begin{proposition}\label{prop:profiles_of_G+}
If $(G_n)_n^\infty$ is Cauchy in $d_M$ and $\limsup_n|V(G_n)|=\infty$ then $(G_n^+)_n^\infty$ is Cauchy in $d_M$ too.
Moreover, if the limiting closures of the $k$-profiles of $(G_n)_n^\infty$ are $X_k:=\lim_{n\to\infty}\overline{\mathcal{S}_k(G_n)}$ then the limiting closures of $\mathcal{S}_k\left(G_n^+\right)$ are $X_k\oplus V_k$ where $V_k=\{0\}^k\times[-1,1]^k$.
\end{proposition}
\begin{proof}
Let us fix $\mu\in X_k$ and $v=\{0\}^k\times(v_1,\dots,v_k)\in V_k$. By the definition of $X_k$, there is a sequence $(\mu_n)_n^\infty=\left(\mathcal{D}_{G_n}(f_1,\dots,f_k)\right)_n^\infty$ of measures in $\mathcal{S}_k(G_n)$ such that $\lim_{n\to\infty}d_{LP}(\mu_n,\mu)=0$. Let $\mu_n^{+v}$ be the measure in $\mathcal{S}_k(G_n^+)$ obtained by assigning the same values to the vertices of the subgraph $G_n$ as $f_1,\dots,f_k$ do and giving the values $(v_1,\dots,v_k)$ to the additional vertex. Formally,
\[\mu_n^{+v}:=\mathcal{D}_{G_n^+}\left(f_1^{+v_1},\dots,f_k^{+v_k}\right)\ \text{ where }\ f_i^{+v_i}(u)=\begin{cases}
f_i(u) &\text{if }u\in V(G_n)\\
v_i &\text{if }\{u\}=V(G_n^+)\setminus V(G_n).
\end{cases}\]
Then
\[d_{LP}(\mu_n^{+v},\mu_n\oplus v)\leq\frac{1}{|V(G_n)|+1}\]
because $(\mu\oplus v)(A\oplus\{v\})=\mu(A)$ for any measurable $A$, and $\left((f_1^{+v_1},\dots,f_k^{+v_k})G_n^+\right)(u)=\left((f_1,\dots,f_k)G_n\right)(u)+(v_1,\dots,v_k)$.
By the triangle inequality,
\[d_{LP}(\mu_n^{+v},\mu\oplus v)\leq d_{LP}(\mu_n^{+v},\mu_n\oplus v)+d_{LP}(\mu_n\oplus v,\mu\oplus v)\leq\frac{1}{|V(G)|+1}+d_{LP}(\mu_n,\mu),\]
where the right-hand side tends to 0 by Lemma~\ref{lemma:verticesToInfinity}.
Hence for every measure in $X_k\oplus V_k$, there is a sequence in $\mathcal{S}_k(G_n^+)$ converging to it.

Vice versa, suppose that $(\mu_n^+=\mathcal{D}_{G_n^+}(f_1^+,\dots,f_k^+))_n^\infty$ is a sequence with $\mu_n^+\in\mathcal{S}_k(G_n^+)$ which is convergent in $d_{LP}$ to a measure $\mu^+$. We want to show that $\mu^+\in X_k\oplus V_k$. Let $v_n\in[-1,1]^k$ be the values assigned by $f_1^+,\dots,f_k^+$ to the added vertex in $G_n^+$. The set $[-1,1]^k$ is compact, so there is a subsequence $\left(\mu^+_{n_i}\right)_i^\infty$ on which $v_n$ converges to some $v\in[-1,1]^k$.

Employing the triangle inequality again gives
\begin{equation}\label{ineq}
d_{LP}(\mu_{n_i}\oplus v,\mu^+)\leq d_{LP}(\mu_{n_i}\oplus v,\mu_{n_i}\oplus v_{n_i})+d_{LP}(\mu_{n_i}\oplus v_{n_i},\mu^+_{n_i})+d_{LP}(\mu^+_{n_i},\mu^+)
\end{equation}
for every $i\in\mathbb{N}$. The second summand on the right-hand side is bounded above by $\frac{1}{|V(G_{n_i})|+1}$ like before and the third summand goes to 0 by the assumption of convergence of $(\mu_n^+)_n^\infty$. To bound the first summand, we observe two facts: that $d_{LP}(\eta\oplus w,\nu)=d_{LP}(\eta,\nu\oplus(-w))$ and that $d_{LP}(\eta,\eta\oplus w)\leq\norm{w}$ for any measures $\eta$ and $\nu$ and vector $w$. The first is true because for any $\varepsilon>0$,
\begin{align*}
&(\eta\oplus w)(U)\leq\nu\left(U^\varepsilon\right)+\varepsilon \text{ for all measurable } U\\
\Leftrightarrow\ 
&(\eta\oplus w)(U\oplus\{w\})\leq\nu\left((U\oplus\{w\})^\varepsilon\right)+\varepsilon \text{ for all measurable } U \\
\Leftrightarrow\ 
&\eta(U)\leq\nu\left(U^\varepsilon\oplus\{w\}\right)+\varepsilon \text{ for all measurable } U \\
\Leftrightarrow\ 
&\eta(U)\leq(\nu\oplus(-w))\left(U^\varepsilon\right)+\varepsilon \text{ for all measurable } U
\end{align*}
and similarly the other way round. The second is true because $U\oplus\{w\}\subseteq U^{\norm{w}+\varepsilon}$ for any $\varepsilon>0$, implying that
\begin{align*}
    \eta(U)=(\eta\oplus w)(U\oplus\{w\})\leq(\eta\oplus w)\left(U^{\norm{w}+\varepsilon}\right)\\
    \text{and } (\eta\oplus w)(U)=\eta(U\oplus\{-w\})\leq\eta\left(U^{\norm{w}+\varepsilon}\right).
\end{align*}
Returning to inequality~\ref{ineq}, we now get
\[
d_{LP}\left(\mu_{n_i},\mu^+\oplus(-v)\right)=d_{LP}\left(\mu_{n_i}\oplus v,\mu^+\right)\leq\norm{v_{n_i}-v}+\frac{1}{|V(G_{n_i})|+1}+d_{LP}(\mu_{n_i}^+,\mu^+)\to0\text{ as } i\to\infty.
\]
Since $\mu_{n_i}\in\mathcal{S}_k(G_{n_i})$, we conclude that $\mu:=\mu^+\oplus(-v)\in X_k$, and so $\mu^+=\mu\oplus v\in X_k\oplus V_k$ as claimed.
\end{proof}

Having established the form of $\lim_n\mathcal{S}_k\left(G_n^+\right)$, we could now show along more technical, but in essence similar lines as in the proof of Proposition~\ref{prop:not_selfadjoint} that no action limit of $(G_n^+)_n^\infty$, where $\lim |V(G_n)|=\infty$, can possibly be self-adjoint.

\begin{notation}
Given a measure space $(\Omega,\mathcal{A},\mu)$, we write $\chi_F$ for the characteristic function of a measurable set $F\in\mathcal{A}$, and $\mathds{1}_E$ for the indicator functions of events $E$ from any other measure space. We also write $B_r^{L^p}(g)$ for the closed $L^p$-ball of radius $r$ around the function $g\in L^p$, and we denote by 0 the function which sends (almost every) element of $\Omega$ to 0.
Finally, $B_1^{L^p}$ is a shorthand for the unit ball around 0.
\end{notation}

We are now preparing to prove that when $G_n\to A$, there indeed is a $P$-operator $A^+$ whose $k$-profiles are like those prescribed by Proposition~\ref{prop:profiles_of_G+} for the prospective limit of $(G_n^+)_n^\infty$. We would like to keep $A$ in $A^+$ in some form because it encodes the possibly complicated structure of $(G_n)_n^\infty$ which is of course also present in $(G_n^+)_n^\infty$. But at the same time we must introduce the shifts arising from the presence of the special vertex in $G_n^+$ which has an outsized influence with respect to its increasingly negligible measure. Like we did in Section~\ref{section:stars}, we will use an ultrafilter-based functional
\begin{align*}
\phi_\mathcal{U} \colon L^\infty&\rightarrow\mathbb{R}\\
\lim_{n\to\infty}\sum_{i=1}^n\alpha_{i,n}\chi_{E_{i,n}}&\mapsto\lim_{n\to\infty}\sum_{i=1}^n\alpha_{i,n}\mathds{1}_{E_{i,n}\in\mathcal{U}}
\end{align*}
for this -- decreasing nested sequences of smaller and smaller sets in $\mathcal{U}$ will play the role that the special vertex played in the proof of Proposition~\ref{prop:profiles_of_G+}. However, we must be careful that the ultrafilter $\mathcal{U}$ does not interfere with the key properties of the original limit $A$. The following theorem, which is also of independent interest, tells us that $\mathcal{U}$ can be chosen to satisfy our needs.

\begin{theorem}\label{thm:FilterMagic}
Let $A\colon L^\infty(\Omega,\mu)\to L^1(\Omega,\mu)$ be a linear operator, where $(\Omega,\mu)$ is an atomless separable finite measure space. Then there is a $\mu$-filter $\mathcal{F}$ such that any ultrafilter $\mathcal{U}$ containing $\mathcal{F}$ has the following property:\\
for all $f\in B_1^{L^\infty}$, $\varepsilon>0$, $a\in[-1,1]$, there is $f_{a,\varepsilon}\in B_1^{L^\infty}$ such that
\begin{enumerate}[label=\Roman*.]
    \item $\norm{f-f_{a,\varepsilon}}_1<\varepsilon$
    \item $\norm{Af-Af_{a,\varepsilon}}_1<\varepsilon$
    \item $\phi_\mathcal{U}(f_{a,\varepsilon})\in(a-\varepsilon,a+\varepsilon)$.
\end{enumerate}
\end{theorem}
\begin{proof}
We will generate $\mathcal{F}$ in countably many steps; using the separability of $L^1(\Omega)$, we will generate countably many functions $f\in L^\infty(\Omega)$ that represent the action of $A$ to an arbitrary precision (see Figure~\ref{figure:covering_by_balls}), and build a filter $\mathcal{F}$ that meshes well with the properties of these functions.

\begin{figure}
    \centering
    \begin{subfigure}[p]{1\textwidth}
 \begin{tikzpicture}
\draw[thick](0,0) ellipse (2.5 and 2.5);
\draw[thick](0,0) ellipse (1.5 and 1.5);
\draw[thick](8,0) ellipse (3 and 3);
\node[color=black, very thick] at (-3.1,1) {$L^\infty(\Omega)$};
\node[color=black, very thick] at (0.9,1.57) {$B_1^{L^\infty}$};
\node[color=black, very thick] at (11.6,1.2) {$L^1(\Omega)$};
    \draw [-stealth](2.74,0.4) -- (4.9,0.4);
\node[color=black, very thick] at (3.9,1) {$A$};
\draw[rotate around={135:(8,0)}, thick](8,0) ellipse (1 and 2);
\node[color=black, very thick] at (8.1,1.79) {$A\left(B_1^{L^\infty}\right)$};

\draw (6.5,-1.2) circle (0.25);
\draw (6.5,-0.9) circle (0.25);
\draw (6.6,-0.5) circle (0.25);
\draw (6.7,-0.1) circle (0.25);
\draw (6.7,-1.3) circle (0.25);
\draw (6.8,0.2) circle (0.25);
\draw (6.8,-1.1) circle (0.25);
\draw (6.9,-0.4) circle (0.25);
\draw (7,-1.4) circle (0.25);
\draw (7,0) circle (0.25);
\draw (7,0.4) circle (0.25);
\draw (7.1,-0.8) circle (0.25);
\draw (7.2,-1.1) circle (0.25);
\draw (7.2,0) circle (0.25);
\draw (7.3,-0.5) circle (0.25);
\draw (7.3,0.7) circle (0.25);
\draw (7.4,0.3) circle (0.25);
\draw (7.5,-1.3) circle (0.25);
\draw (7.5,-0.8) circle (0.25);
\draw (7.6,-0.3) circle (0.25);
\draw (7.6,1) circle (0.25);
\draw (7.7,-0.6) circle (0.25);
\draw (7.7,0.6) circle (0.25);
\draw (7.8,-1.2) circle (0.25);
\draw (7.9,-0.2) circle (0.25);
\draw (8,-0.8) circle (0.25);
\draw (8,0) circle (0.25);
\draw (8,0.5) circle (0.25);
\draw (8,1) circle (0.25);
\draw (8.1,-0.5) circle (0.25);
\draw (8.1,1.2) circle (0.25);
\draw (8.2,0.2) circle (0.25);
\draw (8.2,0.7) circle (0.25);
\draw (8.3,1.1) circle (0.25);
\draw (8.3,-1) circle (0.25);
\draw (8.4,-0.2) circle (0.25);
\draw (8.5,0.9) circle (0.25);
\draw (8.5,1.4) circle (0.25);
\draw (8.5,-0.6) circle (0.25);
\draw (8.6,0.3) circle (0.25);
\draw (8.7,0.6) circle (0.25);
\draw (8.7,-0.3) circle (0.25);
\draw (8.8,0.1) circle (0.25);
\draw (8.9,1.4) circle (0.25);
\draw (9,0.5) circle (0.25);
\draw (9.1,0.7) circle (0.25);
\draw (9,1.2) circle (0.25);
\draw (9.1,-0.2) circle (0.25);
\draw (9.2,0.1) circle (0.25);
\draw (9.3,1.3) circle (0.25);
\draw (9.3,0.4) circle (0.25);
\draw (9.4,1) circle (0.25);
\draw (9.5,0.8) circle (0.25);

\end{tikzpicture}       
        \caption{We first cover $A\left(B_1^{L^\infty}\right)$ with $\frac{1}{n}$-balls, i.e., $A\left(B_1^{L^\infty}\right)$ is represented by the countable collection $\mathcal{C}_n$ up to an error of $1/n$ in the 1-norm.}
    \end{subfigure}
    
    \begin{subfigure}[p]{1\textwidth}
 \begin{tikzpicture}
\draw[thick](0,0) ellipse (2.5 and 2.5);
\draw[thick](0,0) ellipse (1.5 and 1.5);
\draw[thick](8,0) ellipse (3 and 3);
\node[color=black, very thick] at (-3.1,1) {$L^\infty(\Omega)$};
\node[color=black, very thick] at (0.9,1.57) {$B_1^{L^\infty}$};
\node[color=black, very thick] at (11.6,1.2) {$L^1(\Omega)$};
    \draw [-stealth](2.74,0.4) -- (4.9,0.4);
    \draw [stealth-](2.74,-0.4) -- (4.9,-0.4);
\node[color=black, very thick] at (3.9,1) {$A$};
\node[color=black, very thick] at (3.9,-1) {$A^{-1}$};
\draw[rotate around={135:(8,0)}, thick](8,0) ellipse (1 and 2);
\node[color=black, very thick] at (8.1,1.79) {$A\left(B_1^{L^\infty}\right)$};

\draw (6.5,-1.2) circle (0.25);
\draw (6.5,-0.9) circle (0.25);
\draw (6.6,-0.5) circle (0.25);
\draw (6.7,-0.1) circle (0.25);
\draw (6.7,-1.3) circle (0.25);
\draw (6.8,0.2) circle (0.25);
\draw (6.8,-1.1) circle (0.25);
\draw (6.9,-0.4) circle (0.25);
\draw (7,-1.4) circle (0.25);
\draw (7,0) circle (0.25);
\draw (7,0.4) circle (0.25);
\draw (7.1,-0.8) circle (0.25);
\draw (7.2,-1.1) circle (0.25);
\draw (7.2,0) circle (0.25);
\draw (7.3,-0.5) circle (0.25);
\draw (7.3,0.7) circle (0.25);
\draw (7.4,0.3) circle (0.25);
\draw [teal] (7.5,-1.3) circle (0.25);
\draw (7.5,-0.8) circle (0.25);
\draw (7.6,-0.3) circle (0.25);
\draw (7.6,1) circle (0.25);
\node[color=teal, very thick] at (7.98,-1.7) {\small $B_{1/n}^{L^1}\left(g\right)$};
\draw (7.7,-0.6) circle (0.25);
\draw (7.7,0.6) circle (0.25);
\draw (7.8,-1.2) circle (0.25);
\draw (7.9,-0.2) circle (0.25);
\draw (8,-0.8) circle (0.25);
\draw (8,0) circle (0.25);
\draw (8,0.5) circle (0.25);
\draw (8,1) circle (0.25);
\draw (8.1,-0.5) circle (0.25);
\draw (8.1,1.2) circle (0.25);
\draw (8.2,0.2) circle (0.25);
\draw (8.2,0.7) circle (0.25);
\draw (8.3,1.1) circle (0.25);
\draw (8.3,-1) circle (0.25);
\draw (8.4,-0.2) circle (0.25);
\draw (8.5,0.9) circle (0.25);
\draw (8.5,1.4) circle (0.25);
\draw (8.5,-0.6) circle (0.25);
\draw (8.6,0.3) circle (0.25);
\draw (8.7,0.6) circle (0.25);
\draw (8.7,-0.3) circle (0.25);
\draw (8.8,0.1) circle (0.25);
\draw (8.9,1.4) circle (0.25);
\draw (9,0.5) circle (0.25);
\draw (9.1,0.7) circle (0.25);
\draw (9,1.2) circle (0.25);
\draw (9.1,-0.2) circle (0.25);
\draw (9.2,0.1) circle (0.25);
\draw (9.3,1.3) circle (0.25);
\draw (9.3,0.4) circle (0.25);
\draw (9.4,1) circle (0.25);
\draw (9.5,0.8) circle (0.25);

\draw [teal] plot [smooth cycle] coordinates {(-0.8,-1) (0.2,-0.7) (1.2,0.1) (1.4,-0.3) (1.7,-0.5) (1.2,-0.8) (1.4,-1.3) (1.2,-1.4) (0.2,-1.1) (-0.5,-1.3)};
\node[color=teal, very thick] at (0.26,-1.8) {\small $A^{-1}\left(B_{1/n}^{L^1}(g)\right)$};

\draw (-0.67,-1.03) circle (0.25);
\draw (-0.4,-1.1) circle (0.25);
\draw (-0.08,-0.84) circle (0.25);
\draw (-0.04,-1.04) circle (0.25);
\draw (0.3,-0.89) circle (0.25);
\draw (0.39,-1) circle (0.25);
\draw (0.5,-0.79) circle (0.25);
\draw (0.65,-0.6) circle (0.25);
\draw (0.7,-1.06) circle (0.25);
\draw (0.9,-0.39) circle (0.25);
\draw (0.93,-0.99) circle (0.25);
\draw (1.06,-0.77) circle (0.25);
\draw (1.17,-0.15) circle (0.25);
\draw (1.25,-0.42) circle (0.25);

\end{tikzpicture}       
        \caption{Each function $f\in B_1^{L^\infty}$ belongs to a preimage $A^{-1}\left(B_{1/n}^1(g)\right)$, which can itself be represented up to a $\frac{1}{n}$-error in the 1-norm by a countable collection $\mathcal{D}_{n,g}$.}
    \end{subfigure}
\caption{Constructing the countable family $\mathcal{B}_n=\bigcup_{g\in\mathcal{C}_n}\mathcal{D}_{n,g}$ of functions in the unit ball $B_1^{L^\infty}$ whose $\frac{1}{n}$-balls cover it}
\label{figure:covering_by_balls}
\end{figure}

Let us first consider the image $A\left(B_1^{L^\infty}\right)\subseteq L^1(\Omega,\mu)$ of the unit ball of $L^\infty(\Omega)$ under the action of the operator $A$. Since $L^1(\Omega,\mu)$ is separable, we can pick, for any positive integer $n$, a countable collection $\mathcal{C}_n$ of functions $g\in L^1(\Omega,\mu)$ such that the union of their $\frac{1}{n}$-balls covers $A\left(B_1^{L^\infty}\right)$, that is,
\[A\left(B_1^{L^\infty}\right)\subseteq\bigcup_{g\in\mathcal{C}_n}B_{1/n}^{L^1}(g).\]
Next, we look at the preimages $A^{-1}\left(B_{1/n}^{L^1}(g)\right)\subseteq L^\infty(\Omega,\mu)$ of these covering $\frac{1}{n}$-balls.
Given one such fixed preimage $A^{-1}\left(B_{1/n}^{L^1}(g)\right)$, we view it as a subset of $L^1(\Omega,\mu)$ and pick a countable collection $\mathcal{D}_{n,g}$ of functions in $A^{-1}\left(B_{1/n}^{L^1}(g)\right)\cap B_1^{L^\infty}$ whose $\frac{1}{n}$-balls cover $A^{-1}\left(B_{1/n}^{L^1}(g)\right)\cap B_1^{L^\infty}$.
Altogether, this produces the countable family $\mathcal{B}_n=\bigcup_{g\in\mathcal{C}_n}\mathcal{D}_{n,g}$ of functions from the unit ball of $L^\infty(\Omega,\mu)$ which by construction satisfies that for every $f\in B_1^{L^\infty}$, there is some $h_f\in\mathcal{B}_n$ such that both $\norm{h_f-f}_1\leq\frac{1}{n}$ and $\norm{Ah_f-Af}_1\leq\frac{2}{n}$.

We will now construct a filter $\mathcal{F}$ based on the countable collection $\mathcal{B}=\bigcup_{n=1}^\infty\mathcal{B}_n$. Let us start by fixing an ordering $f_1,f_2,\dots$ of the elements of $\mathcal{B}\subset B_1^{L^\infty}$.

For each $m\geq1$, we will find a sequence $E_{m,1}\supset E_{m,2}\supset E_{m,3}\dots$ of measurable sets such that
\begin{enumerate}
    \item $\mu(E_{m,k})>0$ for all $k$, but $\mu(E_{m,k})\to0$ as $k\to\infty$,
    \item $\norm{A\chi_{E_{m,k}}}_1\to0$ as $k\to\infty$, and
    \item the range of $f_m|_{E_{m,k}}$ shrinks, i.e. $\esssup(f_m|_{E_{m,k}})-\essinf(f_m|_{E_{m,k}}) \to0$ as $k\to\infty$.
\end{enumerate}
The conditions 1.--3. will help us to prove the corresponding requirements I.--III. from the statement of the theorem. In particular, condition 3. helps us track the value of $\phi_\mathcal{U}(f_m)$, which we will be able to change by adding functions which are non-zero only on $E_{m,k}$.
\begin{figure}
\centering
\[
\begin{array}{@{}c@{\;}c@{\;}c@{\;}c@{\;}c@{\;}c@{\;}c@{}}
E_{1,1} & \supset & E_{1,2} & \supset & E_{1,3} & \supset &\dots \\
\upin && \upin && \upin &&\\
E_{2,1} & \supset & E_{2,2} & \supset & E_{2,3}& \supset &\dots\\
\upin && \upin && \upin &&\\
E_{3,1} & \supset & E_{3,2} & \supset & E_{3,3}& \supset &\dots\\
\vdots && \vdots && \vdots &&
\end{array}
\]
\caption{The measurable sets $E_{m,k}$ satisfy that for any fixed $m\geq1$, the real numbers $\mu(E_{m,k})$, $\norm{A\chi_{E_{m,k}}}_1$, and $\sup(f_m|_{E_{m,k}})-\inf(f_m|_{E_{m,k}})$ all tend to zero as $k$ tends to infinity.}
\label{figure:GridOfSets}
\end{figure}
Moreover, the sets $E_{m,k}$ will also satisfy the inclusions $E_{1,k}\supset E_{2,k}\supset E_{3,k}\dots$ for every positive integer $k$ (see Figure~\ref{figure:GridOfSets}). Together with that $\mu(E_{m,k})>0$ for all $m,k\in\mathbb{N}$, this implies that
\begin{equation}\label{inequality}
\mu(E_{m_1,k_1}\cap\dots\cap E_{m_n,k_n})>0
\end{equation}
for any $n\in\mathbb{N}$ and $m_1,\dots,m_n,k_1,\dots,k_n\in\mathbb{N}$ because $E_{m_1,k_1}\cap\dots\cap E_{m_n,k_n}\supseteq E_{\max_{i\in[n]}\{m_i\},\max_{i\in[n]}\{k_i\}}$. Finally, inequality ($\ref{inequality}$) tells us that defining
\[
\mathcal{F}:=\left\{\left(E_{m_1,k_1}\cap\dots\cap E_{m_n,k_n}\right)\cup B : n\in\mathbb{N}, B\subseteq\Omega\text{ is measurable}\right\}
\]
gives a $\mu$-filter.

Let us now construct the sequences $(E_{m,k})_{k=1}^\infty$ by induction as follows.
For $m=1$, put
\[
E_{1,1}=\begin{cases}
    f_1^{-1}\left([0,1]\right), & \text{if } \mu\left(f_1^{-1}([0,1])\right)>0\\
    f_1^{-1}\left([-1,0)\right), & \text{if } \mu\left(f_1^{-1}([0,1])\right)=0.
	 \end{cases}
\]
Note that $\mu\left(f_1^{-1}([-1,0))\cup f_1^{-1}([0,1])\right)=\mu(\Omega)$ because $f_1\in B_1^{L^\infty}$, so we necessarily have $\mu(E_{1,1})>0$.
Having obtained $E_{1,k}$ as a subset of $f_1^{-1}\left([x_{1k},x_{1k}+2^{1-k}]\right)$
where $x_{1k}$ is some number in $\left[-1,1-2^{1-k}\right]$, we obtain $E_{1,k+1}$ by first restricting to $R$, where $R$ is given by
\[
R=\begin{cases}
E_{1,k}\cap f_1^{-1}\left([x_{1k},x_{1k}+2^{-k}]\right), & \text{if }\mu\left(E_{1,k}\cap f_1^{-1}([x_{1k},x_{1k}+2^{-k}])\right)>0\\
E_{1,k}\cap f_1^{-1}\left([x_{1k}+2^{-k},x_{1k}+2^{1-k}]\right), & \text{if }\mu\left(E_{1,k}\cap f_1^{-1}([x_{1k},x_{1k}+2^{-k}])\right)=0,
\end{cases}
\]
and then to $R'$, where $R'$ is any measurable subset of $R$ satisfying that $0<\mu(R')\leq\frac{\mu(R)}{2}$ (see Figure~\ref{figure:first_nested_sequence}).
We then set $x_{1,k+1}$ to be equal to $x_{1k}$ if $R$ was chosen to be a subset of $f_1^{-1}\left([x_{1k},x_{1k}+2^{-k}]\right)$ and to $x_{1k}+2^{-k}$ otherwise.
Restricting to $R$ and then $R'$ ensures, respectively, that
\[
\sup\left(f_1|_{E_{1,k+1}}\right)-\inf\left(f_1|_{E_{1,k+1}}\right)\leq2^{-k}
\, \, \text{ and } \, \, 
\mu\left(E_{1,k+1}\right)\leq\frac{\mu(\Omega)}{2^k}.
\]
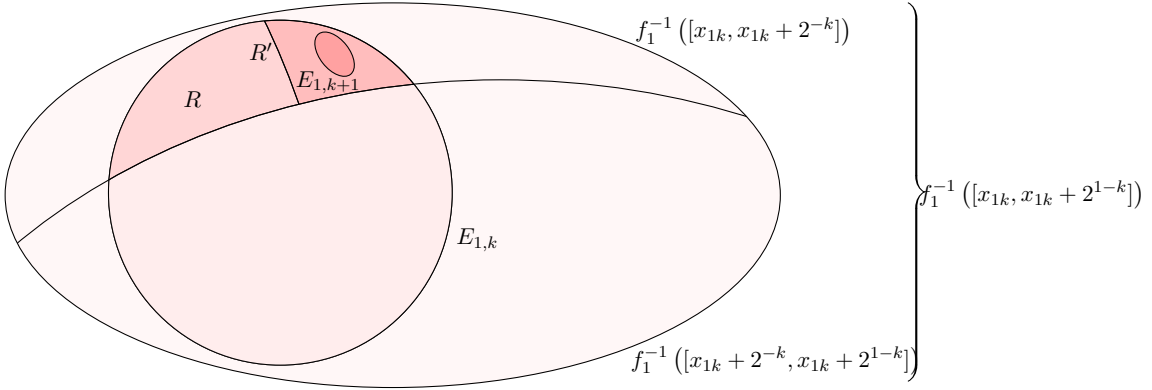
\begin{figure}
\centering
\begin{tikzpicture}[scale=0.85, every node/.style={scale=0.79}]
\filldraw[fill=red!3](0,0) ellipse (6 and 3);
\filldraw[fill=red!26](-1.74,0.04) ellipse (2.66 and 2.69);
\node at (1.31,-0.7) {$E_{1,k}$};

\begin{scope}
  \clip (-1.74,0.04) ellipse (2.66 and 2.69);
  \filldraw[fill=red!16] (-6.66,-3.55) ellipse (6 and 10);
\end{scope}
\begin{scope}
  \clip (-1.74,0.04) ellipse (2.66 and 2.69);
  \draw (-6.66,-3.55) ellipse (6 and 10);
\end{scope}

\node at (5.43,2.56) {$f_1^{-1}\left([x_{1k},x_{1k}+2^{-k}]\right)$};
\node at (5.93,-2.58) {$f_1^{-1}\left([x_{1k}+2^{-k},x_{1k}+2^{1-k}]\right)$};
\draw[decorate, decoration={calligraphic brace, amplitude=6pt, mirror}, thick] (7.98,-2.9) -- (7.98,3);

\node at (9.89,0.01) {$f_1^{-1}\left([x_{1k},x_{1k}+2^{1-k}]\right)$};

\node at (-2.04,2.24) {$R'$};

\begin{scope}
  \clip (-1.74,0.04) ellipse (2.66 and 2.69);
  \filldraw[fill=red!7] (1.67,-6.25) ellipse (10.25 and 8.05);
\end{scope}
\begin{scope}
  \clip (0,0) ellipse (6 and 3);
  \draw (1.67,-6.25) ellipse (10.25 and 8.05);
\end{scope}

\node at (-3.1,1.5) {$R$};

\filldraw[fill=red!37,rotate around={38:(-0.9,2.2)}](-0.9,2.2) ellipse (0.23 and 0.4);
\node at (-1,1.76) {$E_{1,k+1}$};

\draw (-1.74,0.04) ellipse (2.66 and 2.69);
\end{tikzpicture}
\caption{Constructing $E_{1,k+1}$ as a subset of $E_{1,k}$}
\label{figure:first_nested_sequence}
\end{figure}
Since $(\Omega,\mu)$ is atomless, there is an uncountable collection $\mathcal{R}$ of measurable subsets of $R'$ such that for any $S_1\neq S_2\in\mathcal{R}$, we have $\mu(S_1\triangle S_2)>0$ and either $S_1\subset S_2$ or $S_2\subset S_1$. We now consider the uncountable collection of functions $A\chi_S, S\in\mathcal{R}$ in $L^1(\Omega,\mu)$, and conclude that by separability of $L^1(\Omega,\mu)$, there must be some $S\subset T\in\mathcal{R}$ such that $\norm{A\chi_T-A\chi_S}_1<\frac{1}{k}$. Linearity of $A$ then implies that $\norm{A\chi_{T\setminus S}}_1<\frac{1}{k}$, and so we set $E_{1,k+1}:=T\setminus S$.

For $m>1$, suppose we already have $E_{m-1,1}\supset E_{m-1,2}\supset\dots$ as desired. Now if there is some $j\geq1$ such that $\mu\left(E_{m-1,j}\cap f_m^{-1}([0,1])\right)=0$ then we set $E_{m,1}:=E_{m-1,1}\cap f_m^{-1}([-1,0))$, and we necessarily have that $\mu(E_{m,1}\cap E_{m-1,j})>0$ for all $j\in\mathbb{N}$. Otherwise $\mu\left(E_{m-1,j}\cap f_m^{-1}([0,1])\right)>0$ for all $j\in\mathbb{N}$, and we set $E_{m,1}:=E_{m-1,1}\cap f_m^{-1}([0,1])$.
Suppose now that we have obtained $E_{m,k}$ as a subset of $f_m^{-1}\left([x_{mk},x_{mk}+2^{1-k}]\right)$ for some $x_{mk}\in\left[-1,1-2^{1-k}\right]$ and that $\mu(E_{m,k}\cap E_{m-1,j})>0$ for all $j\geq k$. We obtain $E_{m,k+1}$ by first restricting to
\[
R=\begin{cases}
& \mkern-18mu E_{m-1,k+1}\cap E_{m,k}\cap f_m^{-1}\left([x_{mk},x_{mk}+2^{-k}]\right), \\
&\qquad \qquad \qquad \qquad \qquad \quad \quad \quad \quad \text{if }\mu\left(E_{m-1,j}\cap E_{m,k}\cap f_m^{-1}([x_{mk},x_{mk}+2^{-k}])\right)>0 \text{ for all }j>k\\
& \mkern-18mu E_{m-1,k+1}\cap E_{m,k}\cap f_m^{-1}\left([x_{mk}+2^{-k},x_{mk}+2^{1-k}]\right),\\
&\qquad \qquad \qquad \qquad \qquad \quad \quad \quad \quad \text{if }\mu\left(E_{m-1,j}\cap E_{m,k}\cap f_m^{-1}([x_{mk},x_{mk}+2^{-k}])\right)=0 \text{ for some }j>k.
\end{cases}
\]

As before, we set $x_{m,k+1}$ to be $x_{mk}$ if $R$ is chosen to be a subset of $f_m^{-1}\left([x_{mk},x_{mk}+2^{-k}]\right)$ and to be $x_{mk}+2^{-k}$ otherwise.
Let us now partition $R$ into $\bigcup_{j=k+1}^\infty R_j$, where
\[
R_j:=(E_{m-1,j}\setminus E_{m-1,j+1})\cap R.
\]
\begin{figure}
\centering
\begin{tikzpicture}[thick,scale=0.85, every node/.style={scale=0.79}]
\filldraw[color=red!60, fill=red!3, very thick](0,0) ellipse (4 and 3.4);

\node [inner sep=0pt] at (-0.68,-3.48) {\includegraphics[angle=171,width=1.37cm]{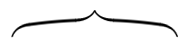} };

\node [inner sep=0pt] at (-1.57,-3.44) {\includegraphics[angle=158,width=3.5cm]{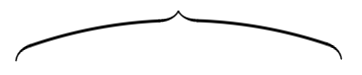} };

\node [inner sep=0pt] at (-2.99,-2.65) {\includegraphics[angle=134,width=6.46cm]{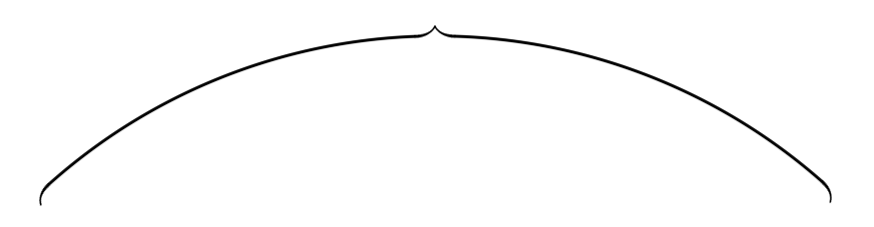} };

\node[color=red!90] at (4.7,-0.16) {$E_{m-1,k+1}$};
\begin{scope}
  \clip (-4,-3.4) rectangle (0,3.4);
  \fill[pattern={Lines[angle=45,distance={4pt},line width={0.1pt}]}] (0,0) ellipse (4 and 3.4);
\end{scope}
\filldraw[color=red!60, fill=red!5, very thick](0,0) ellipse (2.61 and 2.22);
\node[color=red, ultra thick] at (3.3,-0.08) {$E_{m-1,k+2}$};
\begin{scope}
  \clip (-4,-3.4) rectangle (0,3.4);
  \fill[pattern={Lines[angle=135,distance={4pt},line width={0.1pt}]}] (0,0) ellipse (2.61 and 2.22);
\end{scope}
\filldraw[color=red!60, fill=red!5, very thick](0,0) ellipse (1.06 and 0.9);
\node[color=red!90, very thick] at (1.81,0) {$E_{m-1,k+3}$};
\begin{scope}
  \clip (-4,-3.4) rectangle (0,3.4);
  \fill[pattern={Lines[angle=45,distance={4pt},line width={0.1pt}]}] (0,0) ellipse (1.06 and 0.9);
\end{scope}
\filldraw[color=red!60, fill=red!5, very thick](0,0) ellipse (0.16 and 0.14);
\draw[thick, -] (0,3.4) arc (90:270:4 and 3.4);
\draw[thick] (0,3.4) -- (0,-3.4);
    \fill[red!5] (-3.11,1) ellipse (0.42 and 0.21);
    \fill[red!5] (-1.72,0.6) ellipse (0.4 and 0.2);
    \fill[red!5] (-0.56,0.2) ellipse (0.37 and 0.19);
\node[color=black, very thick] at (-3.11,1) {$R_{k+1}$};
\node[color=black, very thick] at (-1.72,0.6) {$R_{k+2}$};
\node[color=black, very thick] at (-0.56,0.2) {$R_{k+3}$};
\draw[thin] (0,0) -- (-4,0);
\draw[thin] (0,0) -- (-2.8,-2.43);
\draw[thin] (0,0) -- (-1.2,-3.24);

\node[color=black, very thick] at (-3.8,-3.38) {\footnotesize $R\left(1/2\right)$};
\node[color=black, very thick] at (-1.72,-3.85) {\footnotesize $R(1/4)$};
\node[color=black, very thick] at (-0.57,-3.69) {\footnotesize $R\left(1/8\right)$};

\draw [teal] plot [smooth cycle] coordinates {(-5.1,-0.6) (-3.9,2.1) (0,3.46) (0.13,0) (1.5,-2) (3,-2.6) (2.62,-4.16) (-1.78,-4.48) (-4.38,-3.5)};
\node[color=teal, very thick] at (3.54,-3.5) {$E_{m,k}$};
\end{tikzpicture}
\caption{Constructing $E_{m,k+1}$ as a subset of $E_{m-1,k+1}$}
\label{figure:R_independent_of_previous_filter}
\end{figure}
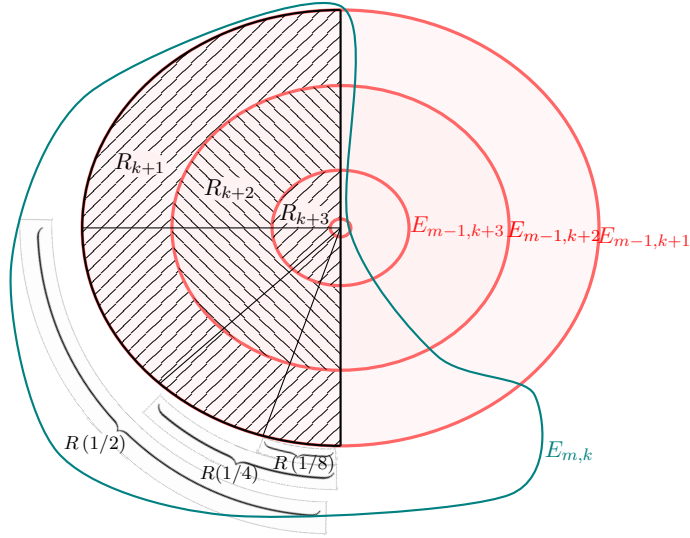
Importantly, we observe that we must have $\mu(R_j)>0$ for infinitely many $j$: if there was some $J>k$ such that $\mu(R_j)=0$ for all $j\geq J$ then we would conclude that
\[
0=\sum_{j=J}^\infty\mu(R_j)=\mu\left(\bigcup_{j=J}^\infty R_j\right)=\mu\left(E_{m-1,J}\cap R\right),
\]
which contradicts the definition of $R$.
Next, by atomlessness of $\Omega$, for every $j>k$, there is a family $\{R_j(t) : t\in(0,1]\}$ of measurable subsets of $R_j$, which satisfies that
\[
\mu(R_j(t))=t\mu(R_j)
\]
and $R_j(t)\subseteq R_j(s)$ whenever $t\leq s$. Let us define $R(t)\subseteq R$ to be
\[
R(t):=\bigcup_{j=k+1}^\infty R_j(t)
\]
for any $t\in(0,1]$, and note that $R(t)\subsetneq R(s)$ whenever $t<s$ (see Figure~\ref{figure:R_independent_of_previous_filter}).
Crucially,
\[\mu\left((R(s)\setminus R(t))\cap E_{m-1,j}\right)>0\]
for all $t<s$ and $j>k$. This is because $(R(s)\setminus R(t))\cap E_{m-1,j}$ contains
\[
\left(R_{j'}(s)\setminus R_{j'}(t)\right)\cap E_{m-1,j'}=R_{j'}(s)\setminus R_{j'}(t)
\]
where $j'\geq j$ is an integer such that $\mu(R_{j'})>0$.

We are now ready to consider the uncountable family $A\chi_{R(t)}$, $t\in(0,1]$ of functions in $L^1(\Omega)$. As before, by separability of $L^1(\Omega)$, there must be some $t<s$ such that $\norm{A\chi_{R(s)}-A\chi_{R(t)}}_1<\frac{1}{k}$.
Linearity of $A$ then implies that
$\norm{A\chi_{R(s)\setminus R(t)}}_1<\frac{1}{k}$,
and we set $E_{m,k+1}:=R(s)\setminus R(t)$.

We have now constructed the sets $E_{m,k}$ as promised and, as already said, we set the $\mu$-filter $\mathcal{F}$ to be
\[
\mathcal{F}:=\left\{\left(E_{m_1,k_1}\cap\dots\cap E_{m_n,k_n}\right)\cup B : n\in\mathbb{N}, B\subseteq\Omega\text{ is measurable}\right\}.
\]
Let now $\mathcal{U}$ be any ultrafilter containing $\mathcal{F}$. It remains to show that for a given $f\in B_1^{L^\infty}$, $\varepsilon>0$ and $a\in[-1,1]$, there is $f_{a,\varepsilon}\in B_1^{L^\infty}$ satisfying the conditions 1.--3.
We first note that for every $m\geq1$, the sequence $(x_{mk})_{k=1}^\infty$ of real numbers converges. Let $x_m$ be the limit, and let us observe that by the careful construction of $\mathcal{F}$, we can conclude that $\phi_\mathcal{U}(f_m)=x_m$ for every $m$. This follows from the general fact that if some measurable set $E$ is in $\mathcal{U}$ then for any function $g\in L^\infty$, the number $\phi_\mathcal{U}(g)$ must be in $\left[\essinf(g|_E),\esssup(g|_E)\right]$. Next, let $n$ be a positive integer such that $\frac{2}{n}<\varepsilon$ and let us recall what we noted at the beginning of the proof, that is that by the construction of $\mathcal{B}_n$, there is some $h_f\in\mathcal{B}_n$ such that $\norm{h_f-f}_1\leq\frac{1}{n}$ and $\norm{Ah_f-Af}_1\leq\frac{2}{n}$. Suppose that $h_f$ appears as $f_m$ in the ordering of $\mathcal{B}=\bigcup_n\mathcal{B}_n$. Then we would like to set $f_{a,\varepsilon}$ to be $f_m+(a-x_m)\chi_{E_{m,k}}$ for some large enough $k$ because linearity of $\phi_\mathcal{U}$ and $E_{m,k}$ being in $\mathcal{F}$ gives us that
\begin{equation}\label{eqn:phiGives_a}
\phi_\mathcal{U}\left(f_m+(a-x_m)\chi_{E_{m,k}}\right)
=\phi_\mathcal{U}(f_m)+(a-x_m)\phi_\mathcal{U}(\chi_{E_{m,k}})
=x_m+(a-x_m)=a.
\end{equation}
However, $f_m+(a-x_m)\chi_{E_{m,k}}$ may not be in the unit ball $B_1^{L^\infty}$, so we will need to do some technical tinkering. On $\Omega\setminus E_{m,k}$, the functions $f_m$ and $f_{a,\varepsilon}$ are equal, so we do not get out of $B_1^{L^\infty}$ there because $f_m=h_f$ is in $B_1^{L^\infty}$ to start with. On $E_{m,k}$, $f_m$ takes values in $\left[x_{mk},x_{mk}+2^{1-k}\right]$, and so $f_m+(a-x_m)\chi_{E_{m,k}}$ takes values in $V=\left[x_{mk}+a-x_m, x_{mk}+2^{1-k}+a-x_m\right]$. But $x_m$ itself is in $\left[x_{mk},x_{mk}+2^{1-k}\right]$, and so $V\subset\left[a-2^{1-k},a+2^{1-k}\right]$. Setting $f_{a,\varepsilon}$ to be
\[
f_{a,\varepsilon}:=\begin{cases}
    f_m+(a-x_m)\chi_{E_{m,k}}, & \text{if } a\in\left[-1+2^{1-k},1-2^{1-k}\right]\\
    f_m+(1-2^{1-k}-x_m)\chi_{E_{m,k}}, & \text{if } a\in\left(1-2^{1-k},1\right]\\
    f_m+(-1+2^{1-k}-x_m)\chi_{E_{m,k}}, & \text{if } a\in\left[-1,-1+2^{1-k}\right)
	 \end{cases}
\]
therefore ensures that $f_{a,\varepsilon}$ is in the unit ball $B_1^{L^\infty}$, and we fix $k$ to be a positive integer large enough so that
\[
\max\left\{\frac{\mu(\Omega)}{2^{k-2}},\frac{2}{k-1},2^{1-k}\right\}<\varepsilon-\frac{2}{n}.
\]
Finally, we check that $f_{a,\varepsilon}$ satisfies the conditions I.--III.
Firstly,
\[
\norm{f-f_{a,\varepsilon}}_1\leq\norm{f-f_m}_1+\norm{f_m-f_{a,\varepsilon}}_1<\frac{1}{n}+2\norm{\chi_{E_{m,k}}}_1=\frac{1}{n}+2\mu(E_{m,k})\leq\frac{1}{n}+2\mu(E_{1,k})\leq\frac{1}{n}+\frac{\mu(\Omega)}{2^{k-2}}<\varepsilon.
\]
Secondly,
\[
\norm{Af-Af_{a,\varepsilon}}_1\leq\norm{Af-Af_m}_1+\norm{Af_m-Af_{a,\varepsilon}}_1\leq\frac{2}{n}+2\norm{A\chi_{E_{m,k}}}_1\leq\frac{2}{n}+\frac{2}{k-1}<\varepsilon.
\]
Thirdly, if $a\in\left[-1+2^{1-k},1-2^{1-k}\right]$ then $\phi_\mathcal{U}(f_{a,\varepsilon})=a$ as per equation (\ref{eqn:phiGives_a}). If $a\in(1-2^{1-k},1]$ then
\[
\phi_\mathcal{U}(f_{a,\varepsilon})=1-2^{1-k}=a+\left(1-2^{1-k}-a\right),
\]
where $1-2^{1-k}-a\in[-2^{1-k},0)$, and similarly for $a\in[-1,-1+2^{1-k})$, so
\[
|\phi_\mathcal{U}(f_{a,\varepsilon})-a|\leq2^{1-k}<\varepsilon
\]
as desired.
\end{proof}

\begin{remark}[on assumptions in Theorem~\ref{thm:FilterMagic}]
In the theorem above, finiteness of $\mu(\Omega)$ is implicitly used in viewing subsets of $L^\infty(\Omega)$ as subsets of $L^1(\Omega)$. Note, however, that we do \emph{not} require $\norm{A}_{\infty\to1}$ to be bounded! One may hope to eliminate the condition of separability of $L^1(\Omega)$, and especially if we would instead insist that $\norm{A}_{\infty\to1}$ be finite, it does not seem unreasonable to believe that this might indeed be possible.
\end{remark}

We now restate and prove our main theorem, of which Theorem~\ref{thm:limitobject} is a special case.

\begin{thmnonumber}
Let $(G_n)_n^\infty$ be a sequence of finite graphs with $|V(G_n)|\to\infty$, whose adjacency operators action converge to a $P$-operator $A\colon L^\infty(\Omega,\nu)\to L^1(\Omega,\nu)$, where $(\Omega,\nu)$ is separable.
Then there is a $\nu$-filter $\mathcal{F}$ on $\Omega$ such that for any $\nu$-ultrafilter $\mathcal{U}$ extending $\mathcal{F}$, both
\begin{equation*}
\begin{aligned}[c]
A^+\colon L^\infty(\Omega,\nu)&\to L^1(\Omega,\nu)\\
\text{given by }\left(A^+g\right)(\omega)&=(Ag)(\omega)+\phi_\mathcal{U}(g)
\end{aligned}
\qquad\text{ and }\qquad
\begin{aligned}[c]
A^-\colon L^\infty(\Omega,\nu)&\to L^1(\Omega,\nu)\\
\text{given by }\left(A^-g\right)(\omega)&=(Ag)(\omega)-\phi_\mathcal{U}(g)
\end{aligned}
\end{equation*}
are action limits of $(G_n^+)_n^\infty$, where $\phi_\mathcal{U}\colon L^\infty(\Omega,\nu)\to\mathbb{R}$ is the functional sending $\lim_{n\to\infty}\sum_{i=1}^n\alpha_{n,i}\chi_{E_{n,i}}$ to $\lim_{n\to\infty}\sum_{i=1}^n\alpha_{n,i}\mathds{1}_{E_{n,i}\in\mathcal{U}}$.
\end{thmnonumber}

As already mentioned, our limits $A^+$ and $A^-$ cannot be self-adjoint, showing that Proposition~\ref{prop:graphops_remain_graphops} (a), which holds under the assumption of uniform boundedness of the $(p,q)$-norms of a Cauchy sequence, cannot be extended to include $(p,q)=(\infty,1)$. However, the situation is interestingly more subtle with being positivity-preserving. If the original limit $A$ is positivity-preserving then so is $A^+$ because the value $\phi_\mathcal{U}(f)$ is always in the essential range of $f$. But that is to say that the functional $-\phi_\mathcal{U}$ is positivity-\emph{reversing}, and so when we return to the case of the star sequence $(S_n)_n^\infty$, where $A\equiv0$ is the trivial $P$-operator on an atomless space $(\Omega,\nu)$, then the limit $A^+$ is positivity-preserving while $A^-$ is not. This shows that Proposition~\ref{prop:graphops_remain_graphops} (b) cannot be extended to $(p,q)=(\infty,1)$, but raises the question whether a Cauchy sequence of graphops always has a positivity-preserving limit (see Section~\ref{section:closeoff}). In any case, we conclude that the property of being positivity-preserving is not invariant under weak equivalence.

Furthermore, if $A$ is $c$-regular for some $c\in\mathbb{R}$ then $A^+$ is $c+1$-regular while $A^-$ is $c-1$-regular. In other words, Theorem~\ref{thm:TheTheorem} shows that also $c$-regularity is not invariant under weak equivalence, unless we first restrict our consideration from the set of all $P$-operators to the set of $P$-operators with the $(p,q)$-norm bounded above by $b$, for some $p,q\in[1,\infty)\times[1,\infty]$ and fixed $b\in\mathbb{R}_{\geq0}$. This in particular means that Proposition~\ref{prop:graphops_remain_graphops} (c) cannot be extended to include $(p,q)=(1,\infty)$.

\begin{proof}
If $\limsup_{n\to\infty}|V(G_n)|<\infty$ then by Lemma~\ref{lemma:verticesToInfinity}, $(G_n)_n^\infty$ is eventually constant, and so is $(G_n^+)_n^\infty$.

Suppose now that $\limsup_{n\to\infty}|V(G_n)|=\infty$.
To prove that $A^+$ and $A^-$ are action limits of $(G_n^+)_n^\infty$, we need to show that the closures of their $k$-profiles equal the limits of the closures of the $k$-profiles of $G_n^+$. In other words, that
\[\overline{\mathcal{S}_k(A^\pm)}=\lim_{n\to\infty}\overline{\mathcal{S}_k(G_n^+)}\] for every $k$. By Proposition~\ref{prop:profiles_of_G+},
\[\lim_{n\to\infty}\overline{\mathcal{S}_k(G_n^+)}=X_k\oplus V_k\]
where $X_k=\lim_n\overline{\mathcal{S}_k(G_n)}$. We will now show first that $\overline{\mathcal{S}_k(A^\pm)}\subseteq X_k\oplus V_k$ and then that $X_k\oplus V_k\subseteq\overline{\mathcal{S}_k(A^\pm)}$.

For a given $k\in\mathbb{N}$, let $\mu=\mathcal{D}_{A^\pm}(f_1,\dots,f_k)$ be any measure in $\mathcal{S}_k(A^\pm)$, given by some $k$ functions in $B_1^{L^\infty}$. Then
\begin{align*}
\mu=\mathcal{D}_{A^\pm}(f_1,\dots,f_k)&=\mathcal{D}\left(f_1,\dots,f_k,A^\pm f_1,\dots,A^\pm f_k\right)\\
&=\mathcal{D}\left(f_1,\dots,f_k,Af_1\pm\phi_\mathcal{U}(f_1)\cdot\mathds{1},\dots,Af_k\pm\phi_\mathcal{U}(f_k)\cdot\mathds{1}\right)\\
&=\mathcal{D}\left(f_1,\dots,f_k,Af_1,\dots,Af_k\right)\oplus(0,\dots,0,\pm\phi_\mathcal{U}(f_1),\dots,\pm\phi_\mathcal{U}(f_k))\in\mathcal{S}_k(A)\oplus V_k.
\end{align*}
This proves that both $\mathcal{S}_k(A^+)$ and $\mathcal{S}_k(A^-)$ are subsets of $\mathcal{S}_k(A)\oplus V_k$, and hence their closures are subsets of $\overline{\mathcal{S}_k(A)\oplus V_k}=X_k\oplus V_k$. In other words, $\overline{\mathcal{S}_k(A^\pm)}\subseteq X_k\oplus V_k$ as required.

On the other hand, let $\mu$ be a measure in $X_k$. Then there is a sequence $\left(\mu_n=\mathcal{D}_A(f^1,\dots,f^k)\right)_n^\infty$ of measures that converge to $\mu$ in $d_{LP}$. The functions $f^1,\dots,f^k\in B_1^{L^\infty}$ naturally depend on $n$, but we again drop this further index in the interest of readability. We also stress that the upper indices are just indices and do not stand for taking powers. Let also $v=\{0\}^k\times(v_1,\dots,v_k)$ be an element of $V_k$. We will now show that we can approximate the measure $\mu\oplus v$ with elements from $\mathcal{S}_k(A^+)$ as well as from $\mathcal{S}_k(A^-)$.

Let $\mathcal{F}$ be a $\nu$-filter given by Theorem~\ref{thm:FilterMagic} and let $\mathcal{U}$ be any ultrafilter extending it. For every $i\in[k]$, let $f^i_{v_i,\frac{1}{n}}$ and $f^i_{-v_i,\frac{1}{n}}$ be the functions also given by Theorem~\ref{thm:FilterMagic} upon applying it to $f^i$. Then we set
\[\mu_n^{+v}:=\mathcal{D}_{A^+}\left(f^1_{v_1,\frac{1}{n}},\dots,f^k_{v_k,\frac{1}{n}}\right)\in\mathcal{S}_k(A^+)\]
and similarly
\[\mu_n^{-v}:=\mathcal{D}_{A^-}\left(f^1_{-v_1,\frac{1}{n}},\dots,f^k_{-v_k,\frac{1}{n}}\right)\in\mathcal{S}_k(A^-).\]
By the triangle inequality,
\begin{align}\label{ineq:triangle_of_mus}
    d_{LP}\left(\mu_n^{\pm v},\mu\oplus v\right)
    &\leq d_{LP}\left(\mu_n^{\pm v},\mu_n\oplus v\right)+d_{LP}\left(\mu_n\oplus v,\mu\oplus v\right)\nonumber\\
    &=d_{LP}\left(\mathcal{D}_{A^\pm}\left(f^1_{\pm v_1,\frac{1}{n}},\dots,f^k_{\pm v_k,\frac{1}{n}}\right),\mathcal{D}_{A}\left(f^1,\dots,f^k\right)\oplus v\right)+d_{LP}\left(\mu_n,\mu\right)
\end{align}
where the second summand goes to 0 by the assumption that $\mu_n\to\mu$.
Shifting our attention to the first summand, we see that for $A^+$ it is equal to {\small
\begin{align*}
    &d_{LP}\left(\mathcal{D}\left(f^1_{v_1,\frac{1}{n}},\dots,f^k_{v_k,\frac{1}{n}},A^+f^1_{v_1,\frac{1}{n}},\dots,A^+f^k_{v_k,\frac{1}{n}}\right),\mathcal{D}\left(f^1,\dots,f^k,Af^1,\dots,Af^k\right)\oplus v\right)\\
    &=d_{LP}\left(\mathcal{D}\left(f^1_{v_1,\frac{1}{n}},\dots,f^k_{v_k,\frac{1}{n}},Af^1_{v_1,\frac{1}{n}}+\phi_\mathcal{U}(f^1_{v_1,\frac{1}{n}})\mathds{1},\dots,Af^k_{v_k,\frac{1}{n}}+\phi_\mathcal{U}(f^k_{v_k,\frac{1}{n}})\mathds{1}\right),\mathcal{D}\left(f^1,\dots,f^k,Af^1+v_1\mathds{1},\dots,Af^k+v_k\mathds{1}\right)\right)
\end{align*}
}and similarly when we start with $\mathcal{D}_{A^-}\left(f^1_{-v_1,\frac{1}{n}},\dots,f^k_{-v_k,\frac{1}{n}}\right)$ instead of $\mathcal{D}_{A^+}\left(f^1_{v_1,\frac{1}{n}},\dots,f^k_{v_k,\frac{1}{n}}\right)$, we arrive to the first summand being
\begin{multline*}
    d_{LP}\left(\mathcal{D}\left(f^1_{-v_1,\frac{1}{n}},\dots,f^k_{-v_k,\frac{1}{n}}, A^-f^1_{-v_1,\frac{1}{n}},\dots, A^-f^k_{-v_k,\frac{1}{n}}\right),\mathcal{D}\left(f^1,\dots,f^k,Af^1,\dots,Af^k\right)\oplus v\right)\\
    =d_{LP}\Big(\mathcal{D}\left(f^1_{-v_1,\frac{1}{n}},\dots,f^k_{-v_k,\frac{1}{n}},Af^1_{-v_1,\frac{1}{n}}-\phi_\mathcal{U}(f^1_{-v_1,\frac{1}{n}})\mathds{1},\dots,Af^k_{-v_k,\frac{1}{n}}-\phi_\mathcal{U}(f^k_{-v_k,\frac{1}{n}})\mathds{1}\right),\\
    \mathcal{D}\left(f^1,\dots,f^k,Af^1+v_1\mathds{1},\dots,Af^k+v_k\mathds{1}\right)\Big).
\end{multline*}

To bound this first summand, both for $A^+$ and $A^-$, we use that by Theorem~\ref{thm:FilterMagic}, the functions $f^i_{v_i,\frac{1}{n}}-f^i$ and $f^i_{-v_i,\frac{1}{n}}-f^i$ as well as $Af^i_{v_i,\frac{1}{n}}-Af^i$ and $Af^i_{-v_i,\frac{1}{n}}-Af^i$ all have their 1-norms bounded above by $1/n$. Markov's inequality then gives that the $2k$ sets
\begin{align*}
N_i^+&=\left\{\omega\in\Omega : \left|\left(f^i_{v_i,\frac{1}{n}}-f^i\right)(\omega)\right|\geq\delta\right\}\\
M_i^+&=\left\{\omega\in\Omega : \left|\left(Af^i_{v_i,\frac{1}{n}}-Af^i\right)(\omega)\right|\geq\delta\right\}
\end{align*}
as well as the $2k$ sets
\begin{align*}
N_i^-&=\left\{\omega\in\Omega : \left|\left(f^i_{-v_i,\frac{1}{n}}-f^i\right)(\omega)\right|\geq\delta\right\}\\
M_i^-&=\left\{\omega\in\Omega : \left|\left(Af^i_{-v_i,\frac{1}{n}}-Af^i\right)(\omega)\right|\geq\delta\right\}
\end{align*}
all satisfy $\nu\left(N_i^\pm\right)\leq\frac{1}{\delta{n}}$ and $\nu\left(M_i^\pm\right)\leq\frac{1}{\delta{n}}$, where $\delta=\delta(n)$ is a positive number to be chosen later. Put together also with that
\[\phi_\mathcal{U}\left(f^i_{v_i,\frac{1}{n}}\right)\in\left(v_i-\frac{1}{n},v_i+\frac{1}{n}\right),\]
this says that the set {\small
\[
M^+:=\left\{\omega\in\Omega : \exists i\in[k]\text{ such that }\left|\left(f^i_{v_i,\frac{1}{n}}-f^i\right)(\omega)\right|\geq\delta\text{ or }\left|\left(A(f^i_{v_i,\frac{1}{n}}-f^i)+(\phi_\mathcal{U}(f^i_{v_i,\frac{1}{n}})-v_i)\mathds{1}\right)(\omega)\right|\geq\delta+\frac{1}{n}\right\}
\]}
satisfies $M^+\subseteq N_1^+\cup\dots\cup N_k^+\cup M_1^+\cup\dots\cup M_k^+$. Analogously, with{\small
\[
M^-:=\left\{\omega\in\Omega : \exists i\in[k]\text{ such that }\left|\left(f^i_{-v_i,\frac{1}{n}}-f^i\right)(\omega)\right|\geq\delta\text{ or }\left|\left(A(f^i_{-v_i,\frac{1}{n}}-f^i)+(-\phi_\mathcal{U}(f^i_{-v_i,\frac{1}{n}})-v_i)\mathds{1}\right)(\omega)\right|\geq\delta+\frac{1}{n}\right\}
\]}
we have that $M^-\subseteq N_1^-\cup\dots\cup N_k^-\cup M_1^-\cup\dots\cup M_k^-$, and so by the union bound,
\[\nu\left(M^\pm\right)\leq\sum_{i=1}^k\nu\left(N_i^\pm\right)+\sum_{i=1}^k\nu\left(M_i^\pm\right)\leq\frac{2k}{\delta{n}}.\]

The inequality above expresses that on most of $\Omega$, the functions
\begin{align*}
\left(f^1_{\pm v_1,\frac{1}{n}},\dots,f^k_{\pm v_k,\frac{1}{n}},Af^1_{\pm v_1,\frac{1}{n}}\pm\phi_\mathcal{U}(f^1_{\pm v_1,\frac{1}{n}})\mathds{1},\dots,Af^k_{\pm v_k,\frac{1}{n}}\pm\phi_\mathcal{U}(f^k_{\pm v_k,\frac{1}{n}})\mathds{1}\right)&\colon\Omega\to\mathbb{R}^{2k}\\
\text{and}\quad\left(f^1,\dots,f^k,Af^1+v_1\mathds{1},\dots,Af^k+v_k\mathds{1}\right)&\colon\Omega\to\mathbb{R}^{2k}
\end{align*}
output real vectors that are close to one another. In particular, for every $\omega\in\Omega\setminus M^\pm$, we have that all the $2k$ coordinates of
$\left(f^1_{\pm v_1,\frac{1}{n}},\dots,Af^k_{\pm v_k,\frac{1}{n}}\pm\phi_\mathcal{U}(f^k_{\pm v_k,\frac{1}{n}})\mathds{1}\right)(\omega)-(f^1,\dots,Af^k+v_k\mathds{1})(\omega)$ are in $(-\delta-1/n,\delta+1/n)$. This implies that
\begin{align*}
    \left(f^1_{\pm v_1,\frac{1}{n}},\dots,Af^k_{\pm v_k,\frac{1}{n}}\pm\phi_\mathcal{U}(f^k_{\pm v_k,\frac{1}{n}})\mathds{1}\right)^{-1}(U)&\subseteq\left(f^1,\dots,Af^k+v_k\mathds{1}\right)^{-1}\left(U^{\norm{(\delta+\frac{1}{n},\dots,\delta+\frac{1}{n})}}\right)\cup M^{\pm}\\
    \text{and}\quad(f^1,\dots,Af^k+v_k\mathds{1})^{-1}(U)&\subseteq\left(f^1_{\pm v_1,\frac{1}{n}},\dots,Af^k_{\pm v_k,\frac{1}{n}}\pm\phi_\mathcal{U}(f^k_{v_k,\frac{1}{n}})\mathds{1}\right)^{-1}\left(U^{\norm{(\delta+\frac{1}{n},\dots,\delta+\frac{1}{n})}}\right)\cup M^{\pm}
\end{align*}
for any measurable subset $U$ of $\mathbb{R}^{2k}$. Taking the measure $\nu$ of both sides of these inclusions gives
\begin{multline*}
    \mathcal{D}\left(f^1_{\pm v_1,\frac{1}{n}},\dots,Af^k_{\pm v_k,\frac{1}{n}}\pm\phi_\mathcal{U}(f^k)\mathds{1}\right)(U)\\
    =\nu\left(\left(f^1_\pm,\dots,Af^k_{\pm v_k,\frac{1}{n}}\pm\phi_\mathcal{U}(f^k_{v_k,\frac{1}{n}})\mathds{1}\right)^{-1}(U)\right)\leq\nu\left((f^1,\dots,Af^k+v_k\mathds{1})^{-1}\left(U^{\norm{(\delta+\frac{1}{n},\dots,\delta+\frac{1}{n})}}\right)\cup M^{\pm}\right)\\
    \qquad\qquad\qquad\qquad\qquad\qquad\qquad\qquad\qquad\qquad\leq\mathcal{D}\left(f^1,\dots,Af^k+v_k\mathds{1}\right)\left(U^{\left(\delta+1/n\right)\sqrt{2k}}\right)+\frac{2k}{\delta{n}}
\end{multline*}
and
\begin{multline*}
\mathcal{D}\left(f^1,\dots,Af^k+v_k\mathds{1}\right)(U)=\nu\left((f^1,\dots,Af^k+v_k\mathds{1})^{-1}(U)\right)\\
\quad\qquad\qquad\qquad\qquad\qquad\leq\nu\left(\left(f^1_{\pm v_1,\frac{1}{n}},\dots,Af^k_{\pm v_k,\frac{1}{n}}\pm\phi_\mathcal{U}(f^k_{v_k,\frac{1}{n}})\mathds{1}\right)^{-1}\left(U^{\norm{(\delta+\frac{1}{n},\dots,\delta+\frac{1}{n})}}\right)\cup M^{\pm}\right)\\
    \leq\mathcal{D}\left(f^1_{\pm v_1,\frac{1}{n}},\dots,Af^k_{\pm v_k,\frac{1}{n}}\pm\phi_\mathcal{U}(f^k_{v_k,\frac{1}{n}})\mathds{1}\right)\left(U^{(\delta+1/n)\sqrt{2k}}\right)+\frac{2k}{\delta{n}},
\end{multline*}
which means that
\[
d_{LP}\left(\mathcal{D}\left(f^1_{\pm v_1,\frac{1}{n}},\dots,Af^k_{\pm v_k,\frac{1}{n}}\pm\phi_\mathcal{U}(f^k_{\pm v_k,\frac{1}{n}})\mathds{1}\right),\mathcal{D}\left(f^1,\dots,Af^k+v_k\mathds{1}\right)\right)
\leq\max\left\{\left(\delta+\frac{1}{n}\right)\sqrt{2k},\frac{2k}{\delta n}\right\}.
\]
Finally, we set $\delta=\delta(n)$ to be $1/\sqrt{n}$, to get that
\begin{multline*}
d_{LP}\left(\mu_n^{\pm v},\mu_n\oplus v\right)=d_{LP}\left(\mathcal{D}\left(f^1_{\pm v_1,\frac{1}{n}},\dots,Af^k_{\pm v_k,\frac{1}{n}}\pm\phi_\mathcal{U}(f^k_{\pm v_k,\frac{1}{n}})\mathds{1}\right),\mathcal{D}(f^1,\dots,Af^k+v_k\mathds{1})\right)\\
\leq\max\left\{\frac{\left(1+\sqrt{n}\right)\sqrt{2k}}{n},\frac{2k}{\sqrt{n}}\right\}\to0.
\end{multline*}
This finishes the argument started at inequality (\ref{ineq:triangle_of_mus}) that $d_{LP}\left(\mu_n^{\pm v},\mu\oplus v\right)\to0$, and so $X_k\oplus V_k\subseteq\overline{\mathcal{S}_k(A^\pm)}$.
\end{proof}

\begin{remark}[on assumptions in Theorem \ref{thm:TheTheorem}]
    The mild requirement that $(\Omega,\nu)$ be separable could be replaced by that $A$ be positivity-preserving or at least by the existence of a measurable set $E$ with $\nu(E)>0$ such that for every $f\geq0$ which is only non-zero on $E$, $Af\geq0$. In that case, the construction of an appropriate $\nu$-filter $\mathcal{F}$ becomes much easier than in the proof of Theorem \ref{thm:FilterMagic}.
\end{remark}

\section{Two questions}\label{section:closeoff}

We tried to to push the boundaries of Theorem \ref{thm:limitobject} and Proposition \ref{prop:graphops_remain_graphops}, and while we concluded that all three parts of the latter fail unless uniform boundedness of $(p,q)$-norms is assumed, we did not find a Cauchy sequence that would exemplify that the norm requirement could not be dropped from the statement of Theorem \ref{thm:limitobject}.

\begin{question}
Let $(A_n)_n^\infty$ be a Cauchy sequence of graphops. Is there always a $P$-operator $A$ with $\lim_n d_M(A_n,A)=0$?
\end{question}

We have also seen limits of graphop sequences which were not positivity-preserving, yet we were always able to find a positivity-preserving limit too.

\begin{question}
Let $(A_n)_n^\infty$ be a graphop sequence which has an action limit $A$. Is there necessarily an action limit $B$ of $(A_n)_n^\infty$ which is positivity-preserving?
\end{question}

\section*{\normalsize Acknowledgement}\noindent
I would like to thank Jan Grebík and Konrad Królicki for discussions that set me on the right track.

\bibliography{knihovna}
\bibliographystyle{plain}

\end{document}